\documentclass[11pt]{amsart}
\usepackage{lscape}
\usepackage{graphicx}
\usepackage{amssymb,amsmath,amsthm,amscd}
\usepackage{mathrsfs}
\usepackage{stmaryrd}
\usepackage{yhmath}
\usepackage{enumerate}
\usepackage{dynkin-diagrams}
\usepackage[colorlinks=true, pdfstartview=FitV,  linkcolor=blue,citecolor=blue,urlcolor=blue, bookmarks=false]{hyperref}
\usepackage[all]{xy}
\usepackage{tikz-cd}

\allowdisplaybreaks[4] 
\numberwithin{equation}{section}

\newtheorem{Theorem}{Theorem}[section] 
\newtheorem{Corollary}[Theorem]{Corollary}
\newtheorem{Lemma}[Theorem]{Lemma}
\newtheorem{Proposition}[Theorem]{Proposition}

\newtheorem{probleminternal}{Problem}

\NewDocumentEnvironment{prob}{o}{%
  \IfValueT{#1}{%
    \renewcommand{\theprobleminternal}{#1}%
    \addtocounter{probleminternal}{-1}%
  }%
  \probleminternal%
}
{\endprobleminternal}

\theoremstyle{definition}
\newtheorem{Definition}[Theorem]{Definition}
\newtheorem{Example}[Theorem]{Example}

\newtheorem{Conjecture}[Theorem]{Conjecture}

\newtheorem{Problem}[Theorem]{Problem}
\newtheorem{Remark}[Theorem]{Remark}
\theoremstyle{remark}

\newcommand{\refb}[1]{(\ref{#1})}

\newcommand{\be}{\begin{enumerate}}

\newcommand{\ene}{\end{enumerate}}

\newcommand{\ZZ}{\mathbb{Z}}

\newcommand{\NN}{\mathbb{N}}
\newcommand{\PP}{\mathbb{P}}

\newcommand{\CC}{\mathbb{C}}

\newcommand{\II}{\mathbb{I}}

\newcommand{\DD}{\mathbb{D}}
\newcommand{\OO}{\mathbb{O}}

\newcommand{\Spec}{\operatorname{Spec}\nolimits}

\newcommand{\End}{\operatorname{End}}
\newcommand{\Hom}{\operatorname{Hom}}
\newcommand{\Ext}{\operatorname{Ext}}

\newcommand{\Ima}{\operatorname{Im}}
\newcommand{\omod}{\operatorname{mod}}

\newcommand{\coker}{\operatorname{coker}}

\newcommand{\IC}{\operatorname{IC}}
\newcommand{\IClam}{\operatorname{IC}(\lambda)}
\newcommand{\ICmu}{\operatorname{IC}(\mu)}
\newcommand{\ICnu}{\operatorname{IC}(\nu)}

\newcommand{\maxs}{\operatorname{max}}
\newcommand{\mins}{\operatorname{min}}

\newcommand{\Rep}{\operatorname{Rep}}
\newcommand{\Res}{\operatorname{Res}}

\newcommand{\Gr}{{\operatorname{Gr}}}
\newcommand{\GL}{{\operatorname{GL}}}
\newcommand{\undd}{\underline{\operatorname{dim}}}
\newcommand{\add}[1]{\operatorname{add}({#1})}
\newcommand{\odim}{\operatorname{dim}}

\newcommand{\Dim}{\operatorname{Dim}}
\newcommand{\soc}{\operatorname{soc}}

\newcommand{\Rproj}{R-\operatorname{proj}}
\newcommand{\Rfin}{R-\operatorname{gmod}}
\newcommand{\KR}{K_0(R)}

\newcommand{\HH}{\operatorname{H}}

\newcommand{\Ch}{\operatorname{Ch}}
\newcommand{\Id}{\operatorname{Id}}

\newcommand{\rank}{\operatorname{rank}}
\newcommand{\ind}{\mathop{\text{\rm ind}}\nolimits}

\newcommand{\KP}[1]{\operatorname{KP}({#1})}
\newcommand{\KQ}[1]{\operatorname{KQ}({#1})}

\newcommand{\Po}[1]{\operatorname{P}_q({#1})}

\newcommand{\fg}{\operatorname{\mathfrak{g}}}

\newcommand{\fn}{\operatorname{\mathfrak{n}}}

\newcommand{\Una}{U_q(\operatorname{\mathfrak{n}})_\mathcal{A}}

\newcommand{\Qp}{Q^+}

\newcommand{\la}{\langle}
\newcommand{\ra}{\rangle}
\newcommand{\wI}[1]{\langle I\rangle_{#1}}

\newcommand{\bd}{{\mathbf{d}}}

\newcommand{\bb}{\mathbf{b}}

\newcommand{\bfi}{\mathbf{i}}
\newcommand{\bfj}{\mathbf{j}}
\newcommand{\bfk}{\mathbf{k}}

\newcommand{\bB}{\mathbf{B}}

\newcommand{\bT}{\mathbf{T}}

\newcommand{\cF}{\mathcal{F}}
\newcommand{\cP}{\mathcal{P}}
\newcommand{\cQ}{\mathcal{Q}}

\newcommand{\cL}{\mathcal{L}}

\newcommand{\cA}{\mathcal{A}}

\newcommand{\cM}{\mathcal{M}}

\newcommand{\cH}{\mathcal{H}}

\newcommand{\cE}{\mathcal{E}}

\newcommand{\eqm}{\stackrel{\sim}{\rightarrow}}
\newcommand{\adb}{{\alpha,\beta}}
\newcommand{\apb}{{\alpha+\beta}}

\newcommand{\tM}[1]{\widetilde{{#1}}}

\newcommand{\cExt}[3]{\cE xt_{#1}({#2},{#3})}

\newcommand{\squw}[2]{[{#1}_{1}{#1}_{2}\cdots{#1}_{#2}]}
\newcommand{\squ}[2]{({#1}_{1},{#1}_{2},\cdots,{#1}_{#2})}

\newcommand{\squg}[2]{{#1}_{i_1}{#1}_{i_2}\cdots {#1}_{#2}}
\newcommand{\squgi}[2]{{#1}_{#2}\cdots {#1}_{i_2} {#1}_{i_1}}

\newcommand{\RepQ}{E_\alpha(Q)}

\newcommand{\Fit}{\tM{\cF}_\bfi}

\newcommand{\fMb}[2]{{\mathfrak{M}^\bullet}({#1},{#2})}
\newcommand{\fMba}[1]{{\mathfrak{M}_0^\bullet}({#1})}
\newcommand{\fMbr}[1]{{\mathfrak{M}_0^{\bullet reg}}({#1})}

\newcommand{\lara}[2]{\la{#1},{#2}\ra}
\newcommand{\Conv}[1]{\Ext_{G_\alpha}^\bullet(\cL_\alpha,{#1})}
\newcommand{\Seteq}[2]{=\left\{{#1}\mid\,{#2}\right\}}
\newcommand{\Set}[2]{\left\{{#1}\mid\,{#2}\right\}}
\newcommand{\simppro}[2]{L({#1})\circ L({#2})}

\newcommand{\ext}{\operatorname{ext}}
\newcommand{\munu}{{\mu,\nu}}

\newcommand{\Aut}[1]{\operatorname{Aut}_Q({#1})}
\newcommand{\inda}{{h\in \Omega}}
\newcommand{\indv}{{i\in I}}

\newcommand{\homq}[2]{\Hom_Q({#1},{#2})}

\newcommand{\extq}[2]{\Ext_Q^1({#1},{#2})}

\newcommand{\bdg}{{\beta,\gamma}}

\title{The product of simple modules over KLR algebras and quiver Grassmannians}
\author{Yingjin Bi}

\address{
	Department of Mathematical Sciences\\ Harbin Engineering University\\ 
	Harbin 150001 P.R.China}
\email{yingjinbi@mail.bnu.edu.cn}
\urladdr{\textrm{\textit{ORCiD}:} \href{https://orcid.org/0000-0003-0153-3274}{orcid.org/0000-0003-0153-3274}}


\begin{document}

\begin{abstract}
In this paper, we study the product of two simple modules over KLR algebras using the quiver Grassmannians for Dynkin quivers. More precisely, we establish a bridge between the Induction functor on the category of modules over KLR algebras and the irreducible components of quiver Grassmannians for Dynkin quivers via a sort of extension varieties, which is an analog of the extension group in Hall algebras. As a result, we give a necessary condition when the product of two simple modules over a KLR algebra is simple using the set of irreducible components of quiver Grassmannians. In particular, in some special cases, we provide proof for the conjecture recently proposed by Lapid and Minguez.
\end{abstract}

\maketitle

\vspace*{6pt}\tableofcontents  

\section{Introduction}
Let $\fg$ be a simple Lie algebra of simply laced type and $U_q(\fn)$ be the half-part of the corresponding quantum group. In \cite{Lu}, Lusztig introduced the canonical basis $\bB$ of $U_q(\fn)$ via the geometrization of $U_q(\fn)$. We call the dual of $\bB$ the dual canonical basis, which is denoted by $\bB^*$. The Hopf algebra $U_q(\fn)$ induces a multiplication on its dual space. We denote by $A_q(\fn)$ the dual space of $U_q(\fn)$. 

It is crucial to study the product of two dual canonical basis elements. For instance, in \cite{BZ} Berenstein and Zelevinsky conjectured that the product of two dual canonical basis elements $\bb_1^*$ and $\bb_2^*$, up to some $q$-power scaling, belongs in $\bB^*$ if and only if $\bb_1^*\bb_2^*=q^n\bb_2^*\bb_1^*$ for some integer $n$. However, Leclerc showed that this conjecture is not true in general, (see \cite{Lec03}), and then proposed the notion of real dual canonical base elements $\bb^*$, \emph{i.e.} $\bb^{*}\bb^*\in q^\ZZ \bB^*$. To modify \cite[Conjecture 1.7]{BZ}, Leclerc proposed the following conjecture.
\begin{Conjecture}{\cite[Conjecture 3.1]{Lec03}}\label{con_lec}
Let $\bb_1^*$ and $\bb_2^*$ be two dual canonical base elements of $\cA_q(\fn)$ such that one of them is real. Suppose that $\bb_1^*\bb_2^*\notin q^\ZZ \bB^*$. The expression of $\bb_1^*\bb_2^*$ is of the following form:
\begin{equation}
  \bb_1^*\bb_2^*=q^m \bb'^{*}+q^s\bb''^*+\sum_{\bb^*\neq \bb'^*,\bb''^*} \gamma_{\bb_1^*,\bb_2^*}^{\bb^*}(q)\bb^*
\end{equation}
where $\bb'^*\neq \bb''^*$ and $m<s\in \ZZ$, and for any $\bb^*\neq \bb'^*,\bb''^*$, we have 
\begin{equation}
  \gamma_{\bb_1^*,\bb_2^*}^{\bb^*}(q)\in q^{m+1}\ZZ[q]\cap q^{s-1}\ZZ[q^{-1}].
\end{equation}
\end{Conjecture}

\noindent
One of the powerful approaches is to categorify of the algebra $U_q(\fn)$ ( or $A_q(\fn)$) using the categories of modules over quiver Hecke algebras $R_Q$, which is introduced by Khovanov and Lauda in \cite{KL09}. This allows us to regard the dual canonical base elements as the simple modules over $R_Q$, see \cite{VV} for more details. 

Based on these results, Kang, Kashiwara, Kim, and Oh use the notion of $R$-matrix to prove Leclerc's conjecture \ref{con_lec}, (see \cite{KKKO}). Meanwhile, Lapid and Minguez proposed the \cite[Conjecture 5.1]{LM2}, which states that $\bb^{\prime *}$ and $\bb^{\prime\prime *}$ are related to the generic extensions of irreducible components of Lusztig's nilpotent varieties. In other words, they study the product of two dual canonical basis elements in terms of representations of preprojective algebras.\\ 

\noindent
However, the $R$-matrices on the tensor product of two simple modules over $R$ can't reveal the explicit knowledge of basis $\bB^*$, such as the combinatorics knowledge of $\bB^*$. The natural indexing of dual canonical basis is the set of Kostant partitions. It is natural to ask the following question: How to describe the condition $\bb_1^*\bb_2^*\in q^\ZZ\bB^*$ in terms of the Kostant partition indexing.

Before starting this, let us review some notations concerning this topic. Let us consider an element $\alpha$ in $\Qp:=\NN[\alpha_i]_{i\in I}$ and denote by $\KP{\alpha}$ the set of Kostant partitions indexing the isomorphic class (up to shifted) of simple modules over the algebra $R_\alpha$, \emph{i.e.} the $\alpha$-part of $R_Q$. We remark that $\KP{\alpha}$ can be thought of as the set of isomorphic class of representations of $Q$ with dimension vector $\alpha$. Here is our question:
\begin{Problem}\label{Pro_main}
Let $\adb$ be two elements in $\Qp$ and $(\munu)\in \KP{\alpha}\times \KP{\beta}$ be a pair of Kostant partitions. Describe the condition on the pair $(\munu)$ such that the Induction $L(\mu)\circ L(\nu)$ of the simple modules $L(\mu)$ and $L(\nu)$ is also a simple module.
\end{Problem}
This question is proposed by Kleshchev and Ram in \cite[Problem 7.6]{KR}. But they used the words of the set of vertices of $Q$ to describe this condition rather than the Kostant partitions.  In type $A_n$, Lapid and Minguez use the combinatorial conditions of multisegments to deal with this problem, see \cite{LM1},\cite{LM3},\cite{LM2} for more details. Inspired by their works, we will use the representation theory of quivers to describe the condition in the Dynkin case. \\

\noindent
The goal of the present paper is to establish a bridge between representations of $Q$ and modules of $R_Q$. More precisely, we will answer Problem \ref{Pro_main} in terms of the irreducible components of quiver Grassmannians for the quiver $Q$. We will also give a sufficient condition when $q^{n}L(\mu\star\nu)\cong\soc \simppro{\mu}{\nu}$ for some integer $n$ in terms of the irreducible components of quiver Grassmannians. In other words, we give a proof for \cite[Conjecture 5.1]{LM2} in some special cases. Let me explain this more explicitly. \\ 

\noindent
First, let $\gamma\in\Qp$ and a partition $\gamma=\apb$, the set $\KP{\gamma}$ is identified with the set of $G_\gamma$-orbits in $E_\gamma$ in Dynkin cases. We define a \emph{partial ordering} on $\KP{\gamma}$ so that $\lambda'\leq\lambda$ if and only if $\OO_\lambda\subset \overline{\OO}_{\lambda'}$ where $\OO_\lambda$ refers to the $G_\gamma$-orbit corresponding to $\lambda$ in $E_\gamma$ and $\overline{\OO}_{\lambda'}$ refers to closure of the orbit $\OO_{\lambda'}$ (see \cite{CB} for more details). For a pair $(\munu)\in\KP{\alpha}\times\KP{\beta}$, one defines $(\mu',\nu')\leq(\munu)$ by $\mu'\leq \mu$ and $\nu'\leq \nu$. 

Let us define the \emph{quiver Grassmannian} of the representation $M_\lambda$ associated with $\lambda\in\KP{\gamma}$. Roughly speaking, it is a variety classifying the subrepresentations of $M_\lambda$ with dimension vector $\beta$. Namely,
\[\Gr_\beta(M_\lambda)\Seteq{W\in \Gr(\bdg)}{\text{ $M_\lambda(W)\subset W$ }}\]
where $\Gr(\bdg)=\prod_{i=1}^n\Gr(\beta_i,\gamma_i)$ is the product of Grassmannians of the vector space $\CC^{\gamma_i}$ with dimension $\beta_i$ for all $i\in[1,n]$. Here we write $\beta=\sum_{i\in I}\beta_i \alpha_i$ and $\gamma=\sum_{i\in I}\gamma_i\alpha_i$. It is well-known that $\Gr_\beta(M_\lambda)$ is a projective variety.\\ 

\noindent
 Let us consider the set of the irreducible components of $\Gr_\beta(M_\lambda)$: Following \cite[Lemma 2.4]{CFR}, we have 
\begin{equation}
  \Gr_\beta(M_\lambda)= \bigcup_{(\munu)\in\KP{\alpha}\times\KP{\beta}} \overline{\Gr(\munu,\lambda)}
\end{equation}
where $\Gr(\munu,\lambda)$ refers to the subvariety consisting of $W\in \Gr_\beta(M_\lambda)$ such that the restriction of $M_\lambda$ to $W$ is isomorphic to $M_\nu$, \emph{i.e.} $(M_\lambda)_{\mid W}\cong M_\nu$ and the restriction of $M_\lambda$ to the quotient space $V/W$ is isomorphic to $M_\mu$, namely, $(M_\lambda)_{\mid V/W}\cong M_\mu$, and $\overline{\Gr(\munu,\lambda)}$ refers to its Zarisky closure. 

In order to describe the set of irreducible components of $\Gr_\beta(M_\lambda)$, we introduce the notion of generic pairs.  Following \cite[Definition 7.3]{CFR1}, we call $(\munu)$ a \emph{generic pair} of $\lambda$ if there exists no $\nu'<\nu$ such that $(M_\lambda)_{W'}\cong M_{\nu'}$ and $(M_\lambda)_{V/W'}\cong M_{\mu}$ for some subspace $W'\in \Gr_\beta(M_\lambda)$ and there exists no $\mu'<\mu$ such that $(M_\lambda)_{W'}\cong M_{\nu}$ and $(M_\lambda)_{V/W'}\cong M_{\mu'}$ for some subspace $W''\in \Gr_\beta(M_\lambda)$. Thanks to Professor Cerulli Irelli and Reineke, we have the following Lemma
\begin{Lemma}{[\textbf{Cerulli Irelli and Reineke}]}
  Under the above assumption, the set of irreducible components of $\Gr_\beta(M_\lambda)$ is identified with the following set
    \begin{multline*}
\ext_{\adb}^{ger}(\lambda)= \left\{(\mu,\nu)\in \KP{\alpha}\times\KP{\beta}\,\middle|\,
\begin{minipage}{.5\linewidth}
 $(\munu)$ is a generic pair of $\lambda$ and\\ 
$\odim \homq{M_\nu}{M_\lambda}
=\odim\homq{M_\nu}{M_\nu}+\odim\homq{M_\nu}{M_\mu}$
\end{minipage}
\right\}
\end{multline*}
\end{Lemma}
Here we introduce the main notion of this paper.
\begin{Definition}
  We call a pair $(\munu)\in\KP{\alpha}\times\KP{\beta}$ a \emph{support pair} if for any $\lambda<\mu\oplus\nu$ we have $(\mu,\nu)\notin\ext_\adb^{ger}(\lambda)$. This notion is inspired by the support of the image of simple perverse sheaf $\IClam$ for $\lambda$ under the Restriction functor. 
\end{Definition}
\noindent
Here is the main result of this paper.
\begin{Theorem}
Let $Q$ be a Dynkin quiver and $L(\mu)$ and $L(\nu)$ be two simple modules over quiver Hecke algebra $R_Q$, where $\mu,\nu$ are their Kostant partitions respectively. If $L(\mu)\circ L(\nu)$ is a simple module, then $(\munu)$ is a support pair. In particular, for any non-trivial extension $\lambda\in\ext(\munu)$, then we have 
\[\odim\Hom_Q(M_\nu, M_\mu\oplus M_\nu)>\odim\Hom_Q(M_\nu,M_\lambda) \text{  if $(\munu)$ is a generic pair for $\lambda$,}\]
and 
\[\odim\Hom_Q(M_\mu,M_\mu\oplus M_\nu)>\odim\Hom_Q(M_\mu,M_\lambda) \text{  if $(\nu,\mu)$ is a generic pair for $\lambda$,}\]
\end{Theorem}
\begin{Remark}
This theorem reveals a relationship between representations of $Q$ and the Induction of two simple modules of $R_Q$. We see that the condition that $\simppro{\mu}{\nu}$ is a simple module is a quite subtle condition. It is difficult to describe this condition in general.\\ 
\end{Remark}

\noindent
Let us go back to Conjecture \ref{con_lec}. Based on \cite[Conjecture 5.1]{LM2}, we consider the generic extension $\mu*\nu$ of any pair $(\munu)\in\KP{\alpha}\times\KP{\beta}$. Namely, the module $M_{\mu*\nu}$ satisfies $M_{\mu*\nu}\in\extq{M_\mu}{M_\nu}$ and $\odim \extq{M_{\mu*\nu}}{M_{\mu*\nu}}$ is minimal with respect to all elements in $\extq{M_\mu}{M_\nu}$. We remark that $\mu*\nu$ always less than or equal to $\mu\star\nu$ in the sense of Lapid and Minguez, but $\mu*\nu=\mu\star\nu$ holds in most of the examples considered in \cite{LM2}. Therefore, the following theorem gives an answer to that conjecture in some special cases.
\begin{Theorem}
Under the above assumption, if $(\munu)\in\ext_{\adb}^{ger}(\mu*\nu)$, then $q^nL(\mu*\nu)$ is a submodule of $L(\mu)\circ L(\nu)$ for some integer $n\in\ZZ$.
\end{Theorem} 
\begin{Remark}
The difficulty is to connect the representations of $Q$ and modules of $\Lambda_Q$. After obtaining this connection, it is possible to give an answer to that conjecture. Although this theorem is not good enough to prove \cite[Conjecture 5.1]{LM2}, this is an attempt to prove this conjecture.\\ 
\end{Remark}

\noindent
\textbf{The strategy used in this paper}
\begin{enumerate}
  \item  Firstly, we transform the product $\simppro{\lambda}{\mu}$ into the Induction functor on the indecomposable projective modules over $R_Q$: $\ind_{\adb}^{\apb}\PP_\nu$ via the bilinear paring~\cite[Section 2.5]{KL09}.
  \item  Second, we transform $\ind_{\adb}^{\apb}\PP_\nu$ into the Induction functor on simple perverse sheaves on representation space: $\Res_{\adb}^{\apb}\ICnu$ via the results in~\cite{VV} and ~\cite{McN}.
  \item  Thirdly, we study the geometric properties of  $\Res_{\adb}^{\apb}\ICnu$ given in~\cite{Sch}, ~\cite{Lu} and~\cite{Lubook} using the graded quiver varieties as in ~\cite{LP} and ~\cite{VV1}
  \item  At last we connect the  $\Res_{\adb}^{\apb}\ICnu$ with quiver Grassmannians given in~\cite{Rei} and ~\cite{CFR}.\\ 
\end{enumerate}

\noindent
\textbf{Acknowledgements}
The author is grateful to Professor Qin Fan for many helpful discussions. He would like to thank Professor Bernard Leclerc and Professor Erez Lapid for useful comments and suggestions; He really appreciates Professor Markus Reineke for his instruction and his proof for the set of irreducible components of quiver Grassmannians; He really thanks Bernard Keller for comments on the graded quiver varieties. He would like to thank Professor Cerulli Irelli for his help and guidance during this work. His work was finished during his invitation to the University of Sapienza. The author is supported by the China Scholarships Council. NO.202006040123.

\section{Premise}\label{sec:premise}
In this section, we recall some basic facts on quiver representation theory for Dynkin quivers.
\subsection{Notations}\label{sec:notation}
In this paper, a quiver $Q=(I,\Omega)$ consists of a set of vertices $I$ and a set of arrows $\Omega$. Denote by $n$ the number of vertices if there is no danger of confusion. For an arrow $h$, one denotes its source and target by $s(h)$ and $t(h)$, respectively. For our purposes, we often assume that $Q$ is of Dynkin type. We assume that the vertex set $I$ is endowed with an order so that $k<l$ if there exists an arrow from $i_k$ to $i_l$. 

Let $C_Q$ be the Cartan matrix for $Q$, simply write $C$. 
We write $\alpha_i$ for its simple roots and $Q^+=\ZZ[\alpha_i]_{i\in I}$ for the monoid generated by $\alpha_i$ for all $i\in I$. The matrix $C$ induces a bilinear form on $Q$, which is denoted by $(-,-)$. Set another bilinear form on $\Qp$ by
\begin{equation}
  \lara{\alpha}{\beta}=\sum_{i\in I} \alpha_i\beta_i-\sum_{h\in \Omega} \alpha_{s(h)}\beta_{t(h)}
\end{equation}
It is well known that $(\adb)=\lara{\alpha}{\beta}+\lara{\beta}{\alpha}$. We write $W$ for its Weyl group generated by simple reflections $s_i$ associated with simple roots $\alpha_i$ for any $i\in I$. For any $w\in W$, one denotes by $l(w)$ the length of $w$. $R^+$ denotes the set of positive roots. The ground field in this paper is the complex number field $\CC$, we sometimes write $k$ for simplicity.  

Let $w_0\in W$ be the longest element in $W$. For a reduced expression of $w_0=\squgi{s}{i_m}$ which is source-adapted with $\Omega$, we obtain the positive root set $R^+$ by  
\begin{equation}\label{eq:roots}
\alpha_{i_1},\  s_{i_1}(\alpha_{i_2}),\ \cdots,\  \squg{s}{i_{k-1}}(\alpha_{i_k}),\ \cdots, \  \squg{s}{i_{m-1}}(\alpha_{i_m}).
\end{equation}
Let us denote them by $\beta_k=\squg{s}{i_{k-1}}(\alpha_{i_k})$ for $1\leq k\leq m$ and give them an order so that $\beta_l<\beta_k$ if and only if $l<k$. In other words, a reduced expression yields an ordering on $R^+$. \\

\subsection{Representations of Dynkin quivers}
Let $M=(V,x)$ be a representation of $Q$ where $V$ is a $I$-graded vector space and $x=(x_h)_{h\in \Omega}$ is a tuple of matrices $x_h:\, V_{s(h)}\to V_{t(h)}$. We sometimes write the matrices $x$ for the representation $M$ if there is no danger of confusion. For a representation $M=(V,x)$, one denote by $\undd M$ its dimension vector $\sum_{i\in I}\odim(V_{i}) \alpha_i\in \Qp$. Given two representations $M,N$ of $Q$, we denote by $\Hom_{Q}(M,N)$ the vector space of $Q$-morphisms between $M$ and $N$ and write $[M,N]$ (resp. $[M,N]^1$) for $\odim\Hom_{Q}(M,N)$ (resp. $\odim\Ext^1_{Q}(M,N)$). One has  
\begin{align}\label{prop_eulerform}
  \langle \undd M,\undd N\rangle=[M,N]-[M,N]^1.
\end{align}
For any vertex $i\in I$, we associate the indecomposable projective (resp. injective) representation $P_i$ (resp. $I_i$). 

Denote by $M_\beta$ the indecomposable module corresponding to the root $\beta$. If the order of positive roots is given as before, it follows
\begin{align}\label{eq_homext0}
      \Hom_{Q}(M_{\beta_a},M_{\beta_b})=\Ext_{Q}(M_{\beta_b},M_{\beta_a})=0 \qquad \text{for $\beta_a<\beta_b$}.
\end{align}
\noindent
Let $\alpha\in\Qp$ and $V$ be a $I$-graded space such that $\undd V=\alpha$, one denotes by $G_\alpha$ the group $\prod_{i\in I}\GL(V_i)$ and by $E_\alpha$ the representation space $\oplus_\inda\Hom_k(V_{s(h)},V_{t(h)})$ endowed with a $G_\alpha$ action by $g\cdot x=(g_{t(h)}x_h g_{s(h)}^{-1})_\inda$. For two $I$-graded vector spaces $V,W$, we denote by $\Hom_I(W,V)$ the space $\oplus_{i\in I}\Hom_k(W_i,V_i)$ and by $\Hom_\Omega(W,V)$ the space $\oplus_{h\in\Omega}\Hom_k(W_{s(h)},V_{t(h)})$.\\

\subsection{Kostant Partitions}\label{sec_Kostant}

Consider the set of the $G_\alpha$-orbits in $E_\alpha$, which is denoted by $\KP{\alpha}$. For an element $\lambda\in \KP{\alpha}$, we denote by $\OO_\lambda$ the $G_\alpha$-orbit corresponding to $\lambda$, and write $d_\lambda$ for the dimension of the orbit $\OO_\lambda$.

  One defines an ordering on $\KP{\alpha}$ so that $\lambda'\leq \lambda$ if $\OO_\lambda\subset\overline{\OO}_{\lambda'}$. Let us recall some properties of $\KP{\alpha}$. In Dykin cases, by Gabriel's Theorem, one can write any element $\lambda$ in $\KP{\alpha}$ as 
  \begin{equation}\label{eq_Kostantpa}
\lambda=\squ{\lambda}{s}    
  \end{equation}
where $\lambda_k\in R^+$ for any $k\in[1,s]$ such that $\lambda_1\geq \lambda_2\geq \cdots \geq \lambda_s$ and $\sum_{i=1}^s\lambda_i=\alpha$. We call $\lambda$ as in \refb{eq_Kostantpa} a \emph{Kostant partition} of $\alpha$. Denote by $M_\lambda$ the module corresponding to the Kostant partition $\lambda\in\KP{\alpha}$.

\begin{Example}\label{exa_A}
Consider the type $A_n$, and fix an orientation of $A_n$ as follows: 
\begin{equation}
1\longrightarrow 2\longrightarrow\cdots\longrightarrow n
\end{equation}
The set $R^+$ of positive roots coincides with the set of segments $[a,b]$ such that $a\leq b$. The ordering of $R^+$ is given by
$$ [i,j]<[k,l]\ \text{if $i<k$ or $i=k$ and $j<l$}$$
Any indecomposable representations of $A_n$ are of the form $M[a,b]$ such that $$\undd{} M[a,b]=\sum\limits_{a\leq i\leq b}\alpha_i.$$
\end{Example}

\subsection{Extension of Kostant partitions}
In this section, we study the extension groups of representations of $Q$ in terms of Kostant partitions. 
Let $\adb,\gamma\in\Qp$ such that $\apb=\gamma$, let $\mu\in \KP{\alpha}$ and $\nu\in\KP{\beta}$. We write $\ext(\munu)$ for the subset of $\KP{\gamma}$ consisting of Kostant partitions $\lambda$ such that 
\begin{equation}\label{eq_shortexa}
    0\to M_\nu\to M_\lambda\to M_\mu\to 0
\end{equation}
In other words, $\ext(\munu)$ is the Kostant partitions whose modules in $\extq{M_\mu}{M_\nu}$.

Similarly, for a Kostant partition $\lambda$, we define $\ext_{\adb}(\lambda)$ as the subset of $\KP{\alpha}\times\KP{\beta}$ consisting of pairs $(\munu)$ satisfies condition \refb{eq_shortexa}. 

\begin{Definition}\label{def:KPlambda}
For a $\lambda\in\KP{\gamma}$ and a partition $\gamma=\apb$, we define $\ext_{\lambda}^{\mins}(\adb)$ by
\begin{multline*}
\qquad\quad \ext_{\lambda}^{\mins}(\adb):=\left\{(\mu,\nu)\in \ext_{\adb}(\lambda)\,\middle|\,
\begin{minipage}{.5\linewidth}
there is no pair $(\mu',\nu')\in\ext_{\adb}(\lambda)$
satisfying $(\mu',\nu')<(\mu,\nu)$
\end{minipage}
\right\}
\end{multline*}
Here by $(\mu',\nu')<(\munu)$, we mean the condition $(\mu',\nu')\leq(\munu)$ and $(\mu',\nu')\neq(\munu)$.
\end{Definition}
\begin{Example}
Given a Kostant partition $\lambda\in\KP{\gamma}$. If the representation $M_\lambda$ admits a decomposition $M_\lambda=M_\mu\oplus M_\nu$ such that $\mu\in\KP{\alpha}$ and $\nu\in\KP{\beta}$, then it is easy to see that
  \[(\munu)\in \ext_{\adb}^{\mins}(\lambda)\] 
    However, the elements in $\ext_{\adb}^{\mins}(\lambda)$ don't satisfy $M_\lambda=M_\mu\oplus M_\nu$ in general. For example: Consider the case $A_3$, let $\lambda=[1,3]+[2,2]$ and $\alpha=\alpha_1+\alpha_2$ and $\beta=\alpha_2+\alpha_3$. It is easy to see that $\ext_{\adb}^{\mins}(\lambda)=\{([1,2],[2,3])\}$.
\end{Example}

\section{Hall maps for Representation spaces}
We will next review the Hall maps for representation spaces (see \cite{Sch} for more details). Fix $\alpha,\beta=\sum_{i\in I}m_i\alpha_i\in\Qp$ and let $\gamma=\apb=\sum_\indv n_i\alpha_i$. Given two $I$-graded vector spaces $W\subset V$ such that $\undd V=\gamma$ and $\undd W=\beta$, we define $\Gr(\beta,\gamma)=\prod_\indv \Gr(m_i,n_i)$ where $\Gr(m_i,n_i)$ are the Grassmannians of $n_i$ for $m_i$. There are two descriptions for it: First, we consider it as 
\begin{equation}\label{eq_grassmannian}
  \Gr(\beta,\gamma)\Seteq{W'\subset V}{\undd W'=\beta}
\end{equation}
Secondly, we interpret it as a quotient variety. Set 
\begin{equation}
  \Hom_I(W,V)^0\Seteq{f=(f_i)_\indv\in\Hom_I(W,V)}{\text{ $f$ is injective.}}
\end{equation}
which is an open subset of $\Hom_I(W,V)$. In other words, $f_i\in Mat_{n_i\times m_i}$ with rank $m_i$ for each $\indv$. The $G_\beta$ acts on $\Hom_I(W,V)$ by $g\cdot f=(f_ig_i^{-1})$. There is a canonical map.
\begin{equation}\label{eq:grassman}
 \begin{split}
 \Pi:\, \Hom_I(W,V)^0&\to \Gr(\beta,\gamma)\\
    f&\mapsto f(W)
 \end{split}
 \end{equation} 
 It is easy to see that $\Hom_I(W,V)^0$ is a $G_\beta$-torsor of $\Gr(\beta,\gamma)$. In other words, $\Pi^{-1}(W')\cong G_\beta$ for any $W'\in \Gr(\beta,\gamma)$.

\subsection{Induction maps}
Let us define the variety $E_{\beta,\gamma}$ by 
\[E_{\beta,\gamma}\Seteq{(x,W')\in E_\gamma\times \Gr(\beta,\gamma)}{x(W')\subset W'}.\]
Fix a subspace $W$ of $V$, one defines 
 \begin{equation}\label{eq:Ebdg'}
   E_\bdg'\Seteq{(x,y,f)\in E_\gamma\times E_\beta\times\Hom_I(W,V)^0}{xf-fy=0}
 \end{equation}
 We have a map $\varphi: E_\bdg'\to E_\bdg$ which maps $(x,y,f)$ to $(x,f(W))$. 
 \begin{Lemma}
     The variety $E_\bdg '$ is a $G_\beta$-bundle over $E_\bdg$.
 \end{Lemma}
\begin{proof}
    We first show that $\varphi$ is surjective. Let us write $W'=f(W)$ for some element $f$ of $\Hom_I(W,V)^0$ via \refb{eq:grassman}. The relation $xf(W)\subset f(W)$ induces a unique $y\in E_\beta$ such that $xf=fy$, as $f$ is injective. More precisely, for each arrow $h:i\to j$, let us consider the following diagram
\begin{equation}
  \xymatrix{
  &0\ar[r] & W_i \ar@{.>}[d]^{y_h} \ar[r]^{f_i} & V_i\ar[d]^{x_h}\\ 
  &0\ar[r] & W_j \ar[r]^{f_j}   & V_j
  }
\end{equation}
Each matrix $x_h$ gives rise to a unique matrix $y_h$ such that $x_h f_i=f_jy_h$. Hence, we obtain that $\varphi$ is surjective.

Meanwhile, one can define $G_\beta$-action on $E_\bdg'$ by $g\cdot(x,y,f)=(x,g\cdot y,g\cdot f)$. By the second definition of quiver Grassmannians, it follows that $E_\bdg '$ is a $G_\beta$-torsor of $E_\bdg$ under this action.
\end{proof} 

\noindent
 Let us consider the canonical map 
\begin{equation}\label{eq_qgrassmann}
  \begin{split}
    q: E_\bdg&\to E_\gamma \\ 
       (x,W')&\mapsto x
  \end{split}
\end{equation}
 and 
\begin{equation}
  \begin{split}
    q':\, E_\bdg&\to \Gr(\bdg) \\ 
       (x,W')&\mapsto W'
  \end{split}
\end{equation}
 The map $q$ is a projective map and $q'$ is a vector bundle with fiber 
 $$\mathsf{P}_\Omega(W,V)=\oplus_\inda \mathsf{P}_{m_{s(h)},n_{t(h)}}$$
  where $\mathsf{P}_{m_{s(h)},n_{t(h)}}$ is the subspace of $\Hom_k(V_{s(h)},V_{t(h)})$ with the form 
$
\begin{pmatrix}
 * & *\\
 0 & *
\end{pmatrix}
$ where the block of $0$ is in $Mat_{(n-m)_{t(h)}\times m_{s(h)}}$.  We see that $E_\bdg$ is a smooth variety, and then the fiber of $q$ at $M\in E_\gamma$ is a projective variety.

\begin{Definition}\label{def_quivergrass}
Consider the map \refb{eq_qgrassmann} $q: E_\bdg\to E_\gamma$. For a representation $M\in E_\gamma$, one calls its fiber $q^{-1}(M)$ the \emph{quiver Grassmannian} of $M$ with dimension $\beta$, and denotes it by $\Gr_\beta(M)$. 
\end{Definition}

Next, one defines 
$$E_\bdg^{(1)}=E_{\bdg}\times G_\beta\times G_\alpha$$
We will give another form of this variety. Set 
    \[\cExt{\gamma}{\alpha}{\beta}\Seteq{(x,y,z,f,u)\in E_\bdg'\times\Hom_\Omega(V/W,V/W)\times\Hom_I(V,V/W)}{0\to y\stackrel{f}{\rightarrow}x\stackrel{u}{\rightarrow}z\to 0}\]
where the sequence $0\to y\stackrel{f}{\rightarrow}x\stackrel{u}{\rightarrow}z\to 0$ is exact.

\begin{Lemma}
    We have an isomorphism: $\cExt{\gamma}{\alpha}{\beta}\cong E_\bdg^{(1)}$. 
\end{Lemma}
\begin{proof}
    To prove this, we define the following map by 
\begin{equation}
\begin{split}
    f: \cExt{\gamma}{\alpha}{\beta}&\to E'_\bdg\\ 
     (x,y,z,f,u)&\mapsto (x, y,f)
\end{split}
\end{equation}
We will show that this map is a $G_\alpha$-torsor over $E'_\bdg$. For an element $(x,y,f)\in E'_\bdg$ and an arrow $h:i\to j$, let us consider the following diagram
\begin{equation}\label{eq_zyu}
\xymatrix{
 & 0\ar[r] &W_i\ar[r]^{f_i} \ar[d]^{y_h} &V_i\ar[r]^{u_i} \ar[d]^{x_h}&V_i/f_iW_i\ar@{.>}[d]^{z_h} \ar[r] &0\\
 & 0\ar[r] &W_j\ar[r]^{f_j} &V_j\ar[r]^{u_j} &V_j/f_j W_j \ar[r] &0
}
\end{equation} 
where $u_i,u_j$ are the cokernel of $f_i,f_j$, respectively, and $z_h$ is the unique map induced by $y_h,x_h$. It follows that any element $(x,y,f)\in E_\bdg'$ gives rise to a unique map $(u,z)\in\Hom_I(V,V/f(W))\times \Hom_\Omega(V/f(W),V/f(W))$. 

 Let us take any element $(x,y,z',f,u')\in \cExt{\gamma}{\alpha}{\beta}$, the element $(x,y,f)\in E_\bdg'$ gives rise to a unique tuple $(x,y,z,f,u)$ as before. Since $u$ is the cokernel of $(x,y,f)$, by the above diagram we have that there exists a unique $g\in G_\alpha\subset \Hom_I(V/f(W), V/W)$ such that $u'=gu$. It implies that $z'=g\cdot z$. Therefore, the variety $\cExt{\gamma}{\alpha}{\beta}$ is a $G_\alpha\times G_\beta$-torsor over $E_{\beta,\gamma}$. 

 Let us consider $E_{\bdg}^{(1)}$: For any element $(x, W',\rho_\beta,\rho_\alpha)$, it is easy to see from \refb{eq:Ebdg'} that $(W',\rho_\beta)$ gives rise to a unique $f\in\Hom_I(W,V)^0$ such that $f(W)=W'$ and $y\in E_\beta$ such that $(x,y,f)\in E_{\bdg}'$. 
 The element $\rho_\alpha$ and $(x,y,f)$ yield a unique $u'=\rho_\alpha u$ as in \refb{eq_zyu}. The induced pair $(f,u')$ induces a short exact sequence 
\[0\to y\stackrel{f}{\rightarrow}x\stackrel{u'}{\rightarrow}z\to 0\]
Therefore, we obtain an isomorphic map  
\[\cExt{\gamma}{\alpha}{\beta}\cong E_\bdg^{(1)}\]
Here we remark that we just check the isomorphism between the closed points of $E_{\bdg}^{(1)}$ and that of $\cExt{\gamma}{\alpha}{\beta}$. It is easy to construct an isomorphic map between $\cExt{\gamma}{\alpha}{\beta}\cong E_\bdg^{(1)}$ using $(f,u')$ and \refb{eq:Ebdg'}.\\ 
\end{proof}

\noindent
There is a map $p:\, E_\bdg^{(1)}\to E_\beta\times E_\alpha$ that sends $(x,y,z,f,u)$ to $(y,z)$. Hence, we obtain 
\begin{equation}\label{dia_induction}
  \begin{aligned}
    E_\beta\times E_\alpha\xleftarrow{\quad p\quad } &E_\bdg^{(1)}\xrightarrow{\quad r\quad } E_\bdg \xrightarrow{\quad q\quad } E_\gamma\\ 
     (y,z)\mapsfrom  (x,y&,z,f,u)\mapsto  (x,f(W)) \mapsto  x 
 \end{aligned}
\end{equation}

\begin{Lemma}\label{lem:oribitsimage}
For two orbits $\OO_\mu$ and $\OO_\nu$ in $E_\alpha$ and $E_\beta$, respectively, we have
\[q rp^{-1}(\OO_\mu\times \OO_\nu)=\bigcup_{\lambda\in\ext(\munu)}\OO_\lambda\] 
\end{Lemma}
\begin{proof}
 Let us consider the subset $p^{-1}(\OO_\mu\times \OO_\nu)$. By definition $\cExt{\gamma}{\alpha}{\beta}$, one has $(x,y,z,f,u)\in p^{-1}(\OO_\mu\times \OO_\nu)$ if and only if $y\in\OO_\nu$, $z\in\OO_\mu$ and there exists a short exact sequence
 \[0\to y\stackrel{f}{\rightarrow} x\stackrel{u}{\rightarrow}z\to 0\]
 Since $qr(x,y,z,f,u)=x$, it follows that $x\in q rp^{-1}(\OO_\mu\times \OO_\nu)$ if and only if there exists a short exact sequence as above. That means $x\in \extq{M_\mu}{M_\nu}$. Therefore, this leads to our conclusion.
\end{proof}

Let us consider the following fibre product diagram
\begin{equation}\label{eq_diagrmunul}
\begin{tikzcd}
  \Gr(\munu,M_\lambda) \arrow[d, "i'"] \arrow[r,"q'"] 
   &M_\lambda \arrow[d, "i"]\\
   rp^{-1}(\OO_\mu\times \OO_\nu) \arrow[r, "q"]
    &E_\gamma
\end{tikzcd}
\end{equation} 
where $M_\lambda$ is a point in $E_\gamma$, $i$ is its embedding map, and $\Gr(\munu,M_\lambda)$ is the $q^{-1}(M_\lambda)\cap rp^{-1}(\OO_\mu\times \OO_\nu)$. It follows from equation \refb{dia_induction} that
\begin{equation}\label{eq_grmunulam}
  \Gr(\munu,M_\lambda)\Seteq{W'\in \Gr_\beta(M_\lambda)}{(M_{\lambda})_{W'}\cong M_\nu;\,  (M_{\lambda})_{V/W'}\cong M_\mu}
\end{equation}

\subsection{Restriction maps}
Fix an $I$-graded subspace $W$ of $V$ such that $\undd V=\gamma$ and $\undd W=\beta$. Let $F_{\bdg}$ be the closed subset of $E_{\gamma}$ consisting of representations $y$ such that $y(W)\subset W$. Let $P_I(W,V)\subset GL(\gamma)$ be the parabolic subgroup associated with $W$. We consider the following diagram
\begin{equation}\label{eq:resdiagram}
  E_\alpha\times E_\beta\stackrel{\kappa}{\leftarrow} F_{\bdg}\stackrel{\iota}{\rightarrow} E_\gamma
\end{equation}
where $\kappa(y)=(y_{\mid V_\gamma/W},y_{\mid W})$ and $\iota$ is the closed embedding. Note that $\kappa$ is a vector bundle of rank 
\begin{equation}
  \rank \kappa=\sum_{i\in I}\alpha_i\beta_i-\lara{\alpha}{\beta}
\end{equation}
For simplicity, we write 
$
\begin{pmatrix}
      &N & U\\
      &0 & M
  \end{pmatrix}
$ for the elements in $F_{\beta,\gamma}$, where $N\in E_\beta$, $M\in E_\alpha$, and $U\in \Hom_\Omega(V/W, W)$.

\subsection{The orbits and restriction maps}
Let us consider the restriction of the map $\kappa$ to the subvariety $\overline{\OO}_\lambda\cap F_{\beta,\gamma}$: $\kappa_\lambda: \overline{\OO}_\lambda\cap F_{\beta,\gamma}\to E_\alpha\times E_\beta$. The fiber of $\kappa_\lambda^{-1}(M_\mu,N_\nu)$ at the point $(M_\mu,M_\nu)$ is identified with 
\begin{equation}\label{eq_fibermn}
  \{E\in F_{\beta,\gamma}\mid \,  E=  \begin{pmatrix}
      &M_\nu & U\\
      &0 & M_\nu
  \end{pmatrix}\text{ and } E\geq M_\lambda\}
\end{equation}
In what follows, we write $\cExt{\lambda}{\mu}{\nu}$ for the variety $\kappa_\lambda^{-1}(M_\mu,M_\nu)$. It is a closed subset of $\Hom_\Omega(V/W,W)$.
\begin{Lemma}\label{lem:fiberorder}
If $\lambda'\geq \lambda$ then we have $\cExt{\lambda'}{\mu}{\nu}\subset\cExt{\lambda}{\mu}{\nu}$.
\end{Lemma}
\begin{proof}
 It is easy to see this by $\overline{\OO}_{\lambda'}\subset\overline{\OO}_\lambda$.
\end{proof}

\section{Quiver Grassmannians and Extension varieties}
In this section, we will give some new interpretations of quiver Grassmannians and extension groups. First, let us recall the notion of the generic extension of two representations of $Q$. Following \cite{Rei}, we have
\begin{Definition}{\cite[Definition 2.2]{Rei}}\label{def_genericqo}
Let $M$ and $N$ be two representations of $Q$. We say $M*N$ is a \emph{generic extension} of $M,N$ if $M*N\in\extq{M}{N}$ such that $\odim\extq{M*N}{M*N}$ is minimal with respect to all extensions in $\extq{M}{N}$. If we write $M$ and $N$ for matrices $(V,x)$ and $(W,y)$ respectively, then any extension $E\in \extq{M}{N}$ can be expressed as the matrix block
\begin{equation}\label{eq_formextens}
M_u=
  \begin{pmatrix}
    y  & u \\ 
    0  & x
  \end{pmatrix}
\end{equation}
where $u\in \Hom_\Omega(V,W)$. Let us set 
\begin{equation}
  \Hom^1_Q(M,N):\Seteq{(u,f)\in \Hom_\Omega(V,W)\times \Hom_I(W\oplus V,W\oplus V)}{f\in \homq{M_u}{M_u}}
\end{equation}
Here we regard $\homq{M_u}{M_u}$ as the solution of the equation 
\[M_{u; h}f_i=f_j M_{u;h} \text{ for each arrow $h:i\to j$}\]
where $M_u=(M_{u;h})_\inda$.

It yields a map $q: \Hom^1_Q(M,N)\to \Hom_\Omega(V,W)$ by sending $(u,f)$ to $u$. It is easy to see that $q^{-1}(u)=\homq{M_u}{M_u}$. There is a unique open subvariety of $\Hom_\Omega(V,W)$ such that $\homq{M_u}{M_u}$ has minimal dimension. The notion of a generic extension of $M,N$ exactly means that the element $u$ for $M*N$ in \refb{eq_formextens} is in this open locus.
\end{Definition}

\subsection{Quiver Grassmannians}
Fix a partition $\apb=\gamma$ of a dimension vector $\gamma$ and a subspace $W$ of an $I$-graded vector space $V$ such that $\undd V=\gamma$ and $\undd W=\beta$. Take an element $x_\lambda\in \OO_\lambda$
of the orbit $\OO_\lambda$ for a Kostant partition $\lambda\in\KP{\gamma}$. Let us consider the following set 
\begin{equation}\label{eq:stableFab}
  G_\lambda\Seteq{g\in G_\gamma}{g\cdot x_{\lambda}(W)\subset W}\Seteq{g\in G_\gamma}{g\cdot
\begin{pmatrix}
y & u\\
0  & x
\end{pmatrix}
\in F_{\beta,\gamma}}
\end{equation}
where $x_\lambda=
\begin{pmatrix}
y & u\\
0  & x
\end{pmatrix}
$ for $u\in \Hom_\Omega(W,V/W)$, $y\in E_\beta$, and $x\in E_\alpha$. It is obvious that $P_{\beta,\gamma}\subset G_\lambda$ and 
\begin{equation}
    G_\lambda\cdot x_\lambda=\OO_\lambda\cap F_{\beta,\gamma}
\end{equation}
\begin{Lemma}\label{lem:quivgrassmanniangroup}
For a representation $M_\lambda$ of $Q$ with dimensional vector $\gamma$, we have $$G_\lambda/P_{\beta,\gamma}\cong\Gr_\beta(M_\lambda)$$
\end{Lemma}
\begin{proof}
By Definition \refb{def_quivergrass}, we see that $\Gr_\beta(M_\lambda)\Seteq{W'\in \Gr(\beta,\gamma)}{x_{\lambda}(W')\subset W'}$. Recall Definition \refb{eq_grassmannian} and that any elements in $\Gr(\beta,\gamma)$ are of the form $g\cdot W$ for some $g\in G_\gamma$, we have $x_{\lambda}(g\cdot W)\subset g\cdot W$. In other words, $g^{-1}\cdot x_\lambda(W)\subset W$. It follows that
\begin{equation}\label{eq_equivalentGr}
  g\cdot W\in \Gr_\beta(M_\lambda) \Longleftrightarrow  g^{-1}\in G_\lambda
\end{equation}
  It implies that there is a surjective map $G_\lambda\to \Gr_\beta(M_\lambda)$ sending $g$ to $g^{-1}W$. The stabilizer of $W$ is exactly $P_{\beta,\gamma}$. It leads to that $G_\lambda/P_{\beta,\gamma}\cong\Gr_\beta(M_\lambda)$.
\end{proof}



\noindent
We will give a decomposition of quiver Grassmannians. In what follows, we fix a decomposition $V=W\oplus U$ such that $\undd W=\beta$ and $\undd U=\alpha$, and we write 
$$
x_\lambda=
\begin{pmatrix}
  y &u \\ 
  0 & x
\end{pmatrix}
$$
where $y\in E_\beta,x\in E_\alpha$ and $u\in \Hom_\Omega(U,W)$\\ 

\noindent
Let $(\munu)\in\KP{\alpha}\times\KP{\beta}$ be a pair of Kostant partitions. Take representations $M_\nu=(W,y)$ and $M_\mu=(U,x)$ for Kostant partitions $\nu$ and $\mu$ such that 
\[0\to M_\nu\to M_\lambda\to M_\mu\to 0\]
Let us set 
\begin{equation}\label{eq_Glammunu}
  G_\lambda(x,y)\Seteq{g\in G_\gamma}{g\cdot
\begin{pmatrix}
y & u\\
0  & x
\end{pmatrix}
=
\begin{pmatrix}
y & u'\\
0  & x
\end{pmatrix}
\text{ for $u'\in\Hom_\Omega(U,W)$}
}
\end{equation}
It is easy to see that 
$P_\lambda(\munu):=P_{\beta,\gamma}\cap G_\lambda(\munu)=
\begin{pmatrix}
\Aut{M_\nu} & \Hom_I(U,W)\\
0  & \Aut{M_\mu}
\end{pmatrix}
$.
We will see that the subvariety $G_\lambda(\munu)$ will connect the quiver Grassmannians and extension varieties.\\

\begin{Lemma}\label{lem:dimfiberkappa}
For a representation $M_\lambda$ of $Q$ with a dimensional vector $\gamma$, we write $x_\lambda$ for its matrix form. Given a partition $\apb=\gamma$ and representations $M_\mu=x$ and $M_\nu=y$ of $Q$ with dimensional vector $\alpha$ and $\beta$ respectively. Let $W$ and $U$ be  fixed $I$-graded vector spaces $\beta$ and $\alpha$ respectively.  For any $\mu\oplus\nu\geq \lambda\geq \mu*\nu$, we see that $\cExt{\lambda}{\mu}{\nu}$ admits a dense open subset $\OO_\lambda\cap \cExt{\lambda}{\mu}{\nu}$. In what follows, we denote by $e_\lambda(\munu)$ the dimension of $\cExt{\lambda}{\mu}{\nu}$.
\end{Lemma}
\begin{proof}
For any $v\in  \Hom_\Omega(U,W)$, one defines 
$
M_{v}=
\begin{pmatrix}
y & v\\
0  & x
\end{pmatrix}
$.  Take a representation $\tM{M}=\bigoplus_{\alpha\in R^+}M_\alpha$ and the dimension vector $\tM{V}$ with $\undd \tM{M}=\tM{V}$. Set 
\[\bT\Seteq{(v,f)\in \Hom_\Omega(U,W)\times \Hom_I(V,\tM{V})}{f\in\homq{M_v}{\tM{M}}}\]
It is easy to see that there are two obvious maps
\begin{equation}
\begin{split}
    \Hom_I(V,\tM{V})\xleftarrow{\ \pi\ } &\bT\xrightarrow{\ \psi\ } \Hom_\Omega(U,W)\\
    f \mapsfrom (v&,f) \mapsto v
\end{split}
\end{equation}
It is easy to see that $\psi^{-1}(v)=\homq{M_v}{\tM{M}}$. It is known that $M_v\leq M_{v'}$ if and only if $\odim \psi^{-1}(v)\leq \odim \psi^{-1}(v')$ for any two element $v,v'\in \Hom_\Omega(U,W)$, see \cite{Bon}. 

We can see that $\bT^{\geq k}\Seteq{(v,f)\in \bT}{\odim \psi^{-1}(v)\geq k}$ is a closed subset of $\bT$. Let $\Hom_\Omega^{\geq k}(U,W)$ be the image of $\bT^{\geq k}$ under the map $\psi$. By \cite[Exercise 3.22 in Chapter 2]{H}, we see that $\bT^k\Seteq{(v,f)\in \bT}{\odim\psi^{-1}(v)=k}$ is a dense open subset of $\bT^{\geq k}$. It follows that its image $\Hom_\Omega^k(U,W)$ under the map $\psi$ is a dense subset of $\Hom_\Omega^{\geq k}(U,W)$. 

Taking $k=\odim \homq{M_\lambda}{\tM{M}}$, it follows that $M_v\geq M_\lambda$ if and only if $v\in \Hom_{\Omega}^{\geq k}(U,W)$ and $M_v\cong M_\lambda$ if and only if $v\in \Hom_{\Omega}^{k}(U,W):\Seteq{v\in \Hom_{\Omega}(U,W)}{\odim \psi^{-1}(v)= k}$. 

Meanwhile, $\overline{\OO}_\lambda\cap \kappa_\lambda^{-1}(x,y)$ consisting of elements 
$
\begin{pmatrix}
    &y  &v'\\
    &0  &x
\end{pmatrix}
$ with $v'\in \Hom_\Omega^{\geq k}(U,W)$. In other words, $\overline{\OO}_\lambda\cap \kappa_\lambda^{-1}(x,y)\cong \Hom_\Omega^{\geq k}(U,W)$. Similarly. we get $\OO_\lambda\cap \kappa_\lambda^{-1}(x,y)\cong \Hom_\Omega^k(U,W)$. Because $\OO_\lambda$ is an open subset of $\overline{\OO}_\lambda$, we see that $\OO_\lambda\cap \kappa_\lambda^{-1}(x,y)$ is an open subset of $\overline{\OO}_\lambda\cap \kappa_\lambda^{-1}(x,y)$. As we discussed before, $\Hom_\Omega^{k}(U,W)$ is a dense subset of $\Hom_\Omega^{\geq k}(U,W)$. It follows from the above discussion that $\OO_\lambda\cap \kappa_\lambda^{-1}(x,y)$ is a dense open subset of $\overline{\OO}_\lambda\cap \kappa_\lambda^{-1}(x,y)$.

From \refb{eq_Glammunu}, we see that $g\in G_\lambda(x,y)$ if and only if $u'\in \Hom_\Omega^k(U,W)$, which implies that $G_\lambda(x,y)/\Aut{M_\lambda}\cong \Hom_\Omega^k(U,W)$.
\end{proof}

\begin{Lemma}
Let $\lambda'>\lambda$ for some $\lambda'\leq \mu\oplus\nu$, we have $e_{\lambda'}(\munu)\leq e_\lambda(\munu)$.
\end{Lemma}
\begin{proof}
 Since $\overline{\OO}_{\lambda'}\subset\overline{\OO}_\lambda$, this implies that $\cExt{\lambda'}{\mu}{\nu}\subset\cExt{\lambda}{\mu}{\nu}$. 
\end{proof}

\begin{Lemma}\label{lem_opendense}
    For a partition $\apb=\gamma$ of a dimension vector $\gamma$ and a Kostant partition $\lambda\in \KP{\gamma}$ with $\lambda\in\ext(\alpha,\beta)$, we have that $\OO_\lambda\cap F_\bdg$ is a dense open subvariety of $\overline{\OO}_\lambda\cap F_\bdg$.
\end{Lemma}
\begin{proof}
Let us set 
\[\bT'\Seteq{(z,f)\in F_{\bdg}\times \Hom_I(V,\tM{V})}{f\in \homq{M_z}{\tM{M}}}\]
where $\tM{M}$ refers to the module $\bigoplus_{\alpha\in R^+}M_\alpha$ and $\undd \tM{M}=\undd \tM{V}$. 

Take $d:=[M_\lambda,\tM{M}]$, and let $\phi: \bT'\to F_{\beta}$ be the map sending $(z,f)$ to $z$. It is easy to see $\phi^{-1}(z)=\homq{M_z}{\tM{M}}$,
\[\overline{\OO}_\lambda\cap F_{\beta,\gamma}=\Set{z\in F_{\bdg}}{\odim \phi^{-1}(z)\geq d},\]
and
\[\OO_\lambda\cap F_{\beta,\gamma}=\Set{z\in F_{\bdg}}{\odim \phi^{-1}(z)=d}.\]
Using the proof of Lemma \ref{lem:dimfiberkappa}, we see that $\OO_\lambda\cap F_{\beta,\gamma}$ is a dense subset of $\overline{\OO}_\lambda\cap F_{\beta,\gamma}$. On the other hand, $\OO_\lambda\cap F_{\bdg}$ is an open subvariety of $\overline{\OO}_\lambda\cap F_{\bdg}$. Thus, we prove our claim. 
\end{proof}

\subsection{Irreducible components of quiver Grassmannians}\label{sec_comquivergrass}
In this section, we will consider the set of irreducible components of quiver Grassmannians. Let us fix a partition of $\gamma=\apb$ and a Kostant partition $\lambda\in\KP{\gamma}$. By \cite{CFR}, the irreducible components of $\Gr_\beta(M_\lambda)$ are the closure of the following subvarieties 
\[S_{N}\Seteq{W'\in \Gr_\beta(M_\lambda)}{x_{\lambda\mid W'}\cong N}\]
for some $N\in E_\beta$. Similarly, these irreducible components are the closure of the subvariety
\[T_{M}\Seteq{W'\in \Gr_\beta(M_\lambda)}{x_{\lambda\mid V/W'}\cong M}\]
for some  $M\in E_\alpha$. Therefore, for such a $S_N$, there exists a $M\in E_\alpha$ such that
  $$\overline{S}_N=\overline{T}_M$$ 
  In this case, the subvariety $S_N\cap T_M$ is an open subvariety of the component $\overline{S}_N$. 
\begin{Definition}
  We define a pair $(\munu)$ as a \emph{component pair} of $\lambda$ if $\overline{S}_{M_\nu}=\overline{T}_{M_\mu}$ is a component of $\Gr_\beta(M_\lambda)$. Therefore, the set of irreducible components of $\Gr_{\beta}(M_\lambda)$ is equal to the set of component pairs $(\munu)$ of $\lambda$. We denote it by $\ext_{\adb}^{ger}(\lambda)$.
\end{Definition}
\begin{Remark}\label{rem_genericpair}
Note that for any component pair $(\munu)$ we see that $M_\mu$ is the generic quotient of $M_\lambda$ by $M_\nu$, \emph{i.e.} there exists no $\mu'<\mu$ such that 
$$0\to M_\nu\to M_\lambda\to M_{\mu'}\to 0$$
and $M_\nu$ is the generic subrepresentation of $M_\lambda$ with quotient $M_\mu$, \emph{i.e.} there exists no $\nu'<\nu$ such that 
$$0\to M_{\nu'}\to M_\lambda\to M_{\mu}\to 0$$
We call $(\munu)$ a \emph{generic pair} if $M_\mu$ is the generic quotient of $M_\lambda$ by $M_\nu$ and $M_\nu$ is the generic subrepresentation of $M_\lambda$ with quotient $M_\mu$. Therefore, a component pair is a generic pair. But a generic pair is not a component pair in general. \\ 
\end{Remark}

\noindent
Thanks to Cerulli Irelli and Reineke, the following lemma gives a description of component pairs.
\begin{Lemma}{[\textbf{Cerulli Irelli and Reineke}]}\label{lem_irrcompo}
Let $\Gr_\beta(M_\lambda)$ be the Grassmannian quiver for representation $M_\lambda$, the subvariety $\overline{S}_N$ is an irreducible component of $\Gr_\beta(M_\lambda)$ if and only if 
\begin{equation}\label{eq_NMT}
  [N,M_\lambda]-[N,N]=[N,T]
\end{equation}
here $T$ is the generic quotient of $M_\lambda$ by $N$, (see Remark \ref{rem_genericpair}).
\end{Lemma}
\begin{proof}
Let us assume that $[N,M_\lambda]-[N,N]=[N,T]$. It is well known that $\odim \overline{S}_N=[N,M_\lambda]-[N,N]$ and tangent space at $N$ is $T_N\Gr_\beta(M_\lambda)=\homq{N}{T}$ (see \cite[Lemma 2.3;2.4]{CFR}). Since 
$$[N,T]=\odim T_N\Gr_\beta(M_\lambda)\geq \odim T_N \overline{S}_N\geq \odim \overline{S}_N=[N,M_\lambda]-[N,N]$$
then the equation \refb{eq_NMT} implies that $N$ is a nonsigular point in $\overline{S}_N$. Meanwhile, by \cite[Lemma 3.1]{Hu}, the identity $\odim T_N\Gr_\beta(M_\lambda)=\odim T_N \overline{S}_N$ implies that $\overline{S}_N$ is an irreducible component of $\Gr_\beta(M_\lambda)$. \\

\noindent
If $\overline{S}_N$ is an irreducible component of $\Gr_\beta(M_\lambda)$, let us take a nonsingular point $g\cdot N$ in $\overline{S}_N$. This means $\odim T_{g\cdot N}\overline{S}_N=\odim T_{g\cdot N}\Gr_\beta(M_\lambda)=[N,T]$. The fact that $g\cdot N$ is a nonsingular point implies $\odim  T_{g\cdot N}\overline{S}_N=\odim \overline{S}_N=[N,M_\lambda]-[N,N]$. Hence, we obtain $$[N,M_\lambda]-[N,N]=[N,T]$$
\end{proof}
\noindent
Therefore, we obtain the following theorem.
\begin{Theorem}{[\textbf{Cerulli Irelli and Reineke}]}\label{theo_irrcom}
 Given a partition $\apb=\gamma$ and a Kostant partition $\lambda\in\KP{\gamma}$. Let $M_\lambda$ be the module associated with $\lambda$. We have that the set of irreducible components of $\Gr_\beta(M_\lambda)$ is identified with the following set
  \begin{center}
\begin{multline*}
\ext_{\adb}^{ger}(\lambda)=\left\{(\mu,\nu)\in \KP{\alpha}\times\KP{\beta}\middle|\,
\begin{minipage}{.5\linewidth}
Such that $(\munu)$ is a generic pair of $\lambda$ and
$[M_\nu,M_\lambda]
=[M_\nu,M_\nu\oplus M_\mu]$
\end{minipage}
\right\}
\end{multline*}
\end{center}
The notion of generic pair is given in Remark \ref{rem_genericpair}.
\end{Theorem}
\begin{proof}
It follows from the above lemma.
\end{proof}
\begin{Corollary}\label{cor_irrOOlam}
  Given a partition $\apb=\gamma$ of a dimensional vector $\gamma$ and a Kostant partition $\lambda\in\KP{\gamma}$, the set of irreducible components of $\overline{\OO_\lambda}\cap F_{\bdg}$ is identified with $\ext_\adb^{ger}(\lambda)$.
\end{Corollary}
\begin{proof}
    Because $\Gr_\beta(M_\lambda)=G_\lambda/P_{\bdg}$ and the fact $P_{\bdg}$ is irreducible, the set of irreducible components of $G_\lambda$ is equal to that of quiver Grassmannian $\Gr_\beta(M_\lambda)$. On the other hand, the fact that $\OO_\lambda\cap F_\bdg=G_\lambda\cdot M_\lambda=G_\lambda/\Aut{M_\lambda}$ implies that the set of irreducible components of $\OO_\lambda\cap F_\bdg$ is equal to that of $G_\lambda$. By Lemma \ref{lem_opendense},  we see that the set of irreducible components of $\overline{\OO_\lambda}\cap F_{\bdg}$ is equal to the set of irreducible components of $\Gr_\beta(M_\lambda)$.
\end{proof}

\begin{Example}\label{exa_A3}
In the case $A_3$, let us consider $\lambda=[1,2]+[2,3]$. If $\beta=\alpha_2+\alpha_3$, let $V=V_1\oplus V_2\oplus V_3$ such that $\odim V_1=\odim V_3=1$ with base elements $e_1,e_3$ and $\odim V_2=2$ with base elements $e_2,e'_2$. Set $x_1: e_1\to e_2 $ and $x_2: e'_2\to e_3$. Fixing $W=\lara{e'_2}{e_3}$, we have 
\[\Gr_\beta(M_\lambda)\Seteq{g W\in \Gr(1,2)}{x_2(g e'_2)\subset W_3=V_3}\]
where $g\in \GL_2$.
It is easy to see that $\Gr_\beta(M_\lambda)=\Gr(1,2)\cong \PP^1$. 

If $\beta=\alpha_1+\alpha_2$ and $W=\lara{e_1}{e_2}$, then 
\[\Gr_\beta(M_\lambda)\Seteq{g W\in \Gr(1,2)}{x_1(e_1)\subset g W_2=\langle g e_1\rangle}\]
This means $\Gr_\beta(M_\lambda)=\{pt\}$.\\ 

\noindent
 Let us consider $\Gr_\beta(M_\lambda)$ with $\beta=\alpha_2+\alpha_3$. There are two generic pairs
\[([2,2]+[3,3],[1,1]+[2,2]) \quad \text{ and } \quad ([2,3],[1,2])\]
by 
  \begin{align*}
    0\to S_2\oplus S_3\to &M_\lambda\to S_1\oplus S_2\to 0 \\ 
    0\to M[2,3]\to &M_\lambda\to M[1,2]\to 0
  \end{align*}
 But $\ext_\adb^{ger}(\lambda)=([2,3],[1,2])$. This follows from the equation
\begin{multline}
  [S_2\oplus S_3, S_2\oplus S_3]+[S_2\oplus S_3, S_1\oplus S_2]=2+1
  >[S_2\oplus S_3, M[2,3]\oplus M[1,2]]=2
\end{multline}

\noindent
In the case $A_2$, following \cite[Section 8.3]{CFR1}, for any dimension vector $\gamma=(d_1,d_2)$ the Kostant partitions in $\KP{\gamma}$ are given by the rank $r$ of the matrix $x_1:\CC^{d_1}\to \CC^{d_2}$.  Let $\beta=(e_1,e_2)<(d_1,d_2)$ be a sub-dimension vector of $\gamma$. If $r<e_1-e_2+d_2$, the set of irreducible components of $\Gr_\beta(M_r)$ is identified with 
\[\ext_\adb^{ger}(M_r)\Seteq{a\in \NN}{\maxs\{0,r+e_1-d_1,r-d_2+e_2\}\leq a\leq \mins\{e_1,e_2,r}\}\]
\end{Example}

\section{Graded quiver varieties}\label{sec_quivervariety}
In this section, we will explore the concept of graded quiver varieties and the connections between the space of representations for $Q$ and the graded quiver varieties. 
\subsection{Repetition quivers}
Let's start by delving into the notion of repetition algebras. For a quiver $Q=(I,\Omega)$, we define a \emph{height function} $\xi :\, I\to \ZZ$ on its vertex set such that 
\begin{equation}
    \xi_i=\xi_j+1    \text{  if there exists an arrow $h:i\to j$}
\end{equation}

By fixing a specific height function $\xi$, we can define a set
\begin{equation}
    \widehat{I}\Seteq{(i,p)\in I\times \ZZ}{p-\xi_i\in 2\ZZ}
\end{equation}
We attach to $Q$ the infinite \emph{repetition quiver} $\widehat{Q}$, defined as the quiver with vertex set $\widehat{I}$ and two type arrows:
\begin{align}
     (h,p):\, (i,p)\to (j,p+1)  &\text{ if $h:i\to j$}\\ 
     (\bar{h},q):\, (j,q) \to (i,q+1)  &\text{ if $h:i\to j$}
 \end{align} 
 for all $(i,p),(j,q)\in\widehat{I}$.  $\widehat{Q}$ is recognized as a $\ZZ$-cover of the preprojective algebra associated with $Q$. 

 Consider $\widehat{\Delta}=R^+\times\ZZ$, where $R^+$ denotes the collection of positive roots. We will now explain a natural way to label the vertices of $\widehat{Q}$ using the elements of $\widehat{\Delta}$. Let $\gamma_i$ be the root associated with indecomposable injective representation $I_i$ for each $i\in I$. There exists a unique bijection $\phi :\, \widehat{I}\to \widehat{\Delta}$ defined inductively as follows.
 \begin{enumerate}
      \item $\phi(i,\xi_i)=(\gamma_i,0)$ for $i\in I$.
      \item Suppose $\phi(i,p)=(\beta,m)$ then 
          \begin{align*}
              \phi(i,p-2)&=(\tau\beta,m)  &\text{ if } \tau\beta \in R^+\\ 
               \phi(i,p-2)&=(-\tau\beta,m-1)  &\text{ if } \tau\beta \in R^-\\ 
               \phi(i,p+2)&=(\tau^{-1}\beta,m)  &\text{ if } \tau^{-1}\beta \in R^+\\ 
               \phi(i,p+2)&=(-\tau^{-1}\beta,m+1)  &\text{ if } \tau^{-1}\beta \in R^-
          \end{align*}
          where $\tau$ is the Auslander-Reiten translation for $Q$.
  \end{enumerate} 
 It is worth noting that the Auslander-Reiten quiver $\Gamma_Q$ can be viewed as a subquiver of $\widehat{Q}$ by mapping $M_\beta$ to $\phi^{-1}(\beta,0)$. Let $\widehat{I}_Q=\phi^{-1}(R^+\times \{0\})$, which represents the vertices of $\Gamma_Q$.\\

\noindent
We define the \emph{$q$-analog of the Cartan matrix} $C_q$ on $\NN[I\times\ZZ]$ by 
\begin{equation}\label{eq_cartanq}
    \begin{split}
        C_q: \, \NN[I\times\ZZ]&\to \NN[I\times\ZZ]\\ 
              V(i,p) &\mapsto V(i,p-1)+V(i,p+1)-\sum_{j\sim i}a_{i,j}V(j,p)
    \end{split}
\end{equation}
where $V(i,p)$ is the dimension vector with support at $(i,p)$, and its coefficient at $(i,p)$ is equal to $1$. When it is said that $i\sim j$, it implies that there is an arrow either from $i$ to $j$ or from $j$ to $i$ in the set $\Omega$.

  \subsection{Graded quiver varieties}\label{sec_gradedquivervarieties} 
  Consider a finite dimensional $\widehat{I}$-graded vector space
  \[W=\bigoplus_{(i,p)\in\widehat{I}}W_i(p).\]
  Nakajima \cite{Na}  introduced an affine variety $\mathcal{M}_W$ associated with $W$. More precisely, let 
  $$\widehat{J}\Seteq{(i,p)\in I\times \ZZ}{(i,p-1)\in \widehat{I}},$$
and let  
 \[V=\bigoplus_{(i,p)\in\widehat{J}}V_i(p)\]
 be a finite dimensional $\widehat{J}$-graded vector space. One defines 
 \begin{align}
     L^\bullet(V,W)&=\bigoplus_{(i,p)\in \widehat{J}} \Hom(V_i(p),W_i(p-1))\\ 
     L^\bullet(W,V)&=\bigoplus_{(i,p)\in \widehat{I}} \Hom(W_i(p),V_i(p-1))\\ 
     E^\bullet(V)&=\bigoplus_{(i,p)\in \widehat{J}, j\sim i} \Hom(V_i(p),V_j(p-1))
 \end{align}
 Put $M^\bullet(V,W)=E^\bullet(V)\oplus L^\bullet(W,V)\oplus L^\bullet(V,W)$. An element of $M^\bullet(V,W)$ is written by $(B,\alpha,\beta)$, and its components are denoted by 
  \begin{equation}\label{eq_Bab}
  \begin{split}
         B_{ij}(p)&\in \Hom(V_i(p),V_j(p-1))\\ 
     \alpha_i(p)&\in \Hom(W_i(p),V_i(p-1))\\ 
     \beta_i(p)&\in \Hom(V_i(p),W_i(p-1))
  \end{split}
 \end{equation}
 Let $\Lambda^\bullet(V,W)$ be the subvariety of $M^\bullet(V,W)$ defined by the equations
 \begin{equation}\label{eq_formularelation}
     \alpha_{i}(p-1)\beta_i(p)+\sum_{i\to j} B_{ji}(p-1)B_{ij}(p) -\sum_{k\to i} B_{ki}(p-1)B_{ik}(p)=0
 \end{equation}
 The group $G_V=\prod_{(i,p)\in \widehat{J}}\GL(V_i(p))$ acts on $M^\bullet(V,W)$ by 
 \[g\cdot (B,\alpha,\beta)=(g_j(p-1) B_{ij}(p) g^{-1}_i(p), g_i(p-1)\alpha_i(p), \beta_{i}(p)g^{-1}_i(p))\]
 This action preserves
 $\Lambda^\bullet(V,W)$. One defines the affine variety
 \begin{equation}
     \fMba{V,W}=\Lambda^\bullet(V,W)\sslash G_V.
 \end{equation}
 By definition, the coordinate ring of $\fMba{V,W}$ is the ring of $G_V$-invariant functions on $\Lambda^\bullet(V,W)$. If there is a graded vector space $V'$ such that $V_i(p)\leq V'_i(p)$ for all $(i,p)\in \widehat{J}$, then we have a natural closed embedding $\fMba{V,W}\subset\fMba{V',W}$. Finally, one can define
 \begin{equation}
     \fMba{W}=\bigcup_{V}\fMba{V,W}
 \end{equation}
 Let $\fMbr{V,W}$ be the open subset of $\fMba{V,W}$ parametrizing the closed free $G_V$-orbits. For a given $W$, we have $\fMbr{V,W}\neq \emptyset$ for only a finite number of $V$'s. Nakajima has shown that these subvarieties $\fMbr{V,W}$ give a stratification of $\fMba{V,W}$ as follows.
 \begin{equation}\label{eq_strareg}
     \fMba{W}=\bigsqcup_{V}\fMbr{V,W}
 \end{equation}
 A necessary condition for $\fMbr{V,W}\neq \emptyset$ is that
 \begin{equation}
     \odim W_i(p)-\odim V_i(p+1)-\odim V_i(p-1)+\sum_{i\sim j} V_j(p)\geq 0
 \end{equation}
 for each $(i,p)\in \widehat{I}$. In other words, $\undd W-C_q(\undd V)\geq 0$. In this case, we say that $(V,W)$ is a \emph{dominate pair}. We denote by $\IC_W(V)$  the intersection cohomology complex of the closure of the stratum $\fMbr{V,W}$.

 \subsection{An isomorphism}
 Let $\gamma=\sum_{i=1}^r\gamma_i \alpha_i$. We define an $\widehat{I}$-graded vector space $W^\gamma$ by taking
 \begin{equation}\label{eq_wip}
     W_j(p)=\CC^{\gamma_i}  \text{  if  }  \phi(j,p)=(\alpha_i,0)
 \end{equation}
 and $W_i(p)=0$ for others $(i,p)\in\widehat{I}$.Following \cite[Theorem 9.11]{HG}, the subsequent theorem is as follows.
 \begin{Theorem}[{\cite[Theorem 9.11]{HG}}]\label{theo_desingular}
     We have a $G_\gamma$-equivariant isomorphism $$\Psi: \fMba{W^\gamma}\cong E_\gamma$$ such that 
     \begin{enumerate}
         \item The stratification given by $G_\gamma$-orbit on $E_\gamma$ coincides with the stratification \refb{eq_strareg}. Namely, $\Psi(\fMbr{V,W^\gamma})=\OO_\lambda$ for some $\lambda\in\KP{\gamma}$;
         \item As a corollary, we have $\IClam=\Psi_*(\IC_W(V))$ where $\IC(\lambda)$ is given in Section \ref{sec_lus}.
     \end{enumerate}
 \end{Theorem}
Let us set 
\begin{equation}\label{eq_bilinear}
\begin{split}
    \lara{-}{-}:\, \NN[I\times\ZZ]\times & \NN[I\times\ZZ]\to \ZZ \\ 
              (V,& W)\mapsto \sum_{(i,a)\in I\times\ZZ}V(i,a)W(i,a)
\end{split}
\end{equation}
and
$q^\pm V=(V'(i,p))$ such that $V'(i,p)=V(i,p\pm 1)$. For a pair $(V_1,W_1; V_2,W_2)$ of graded vector spaces, one defines $d(V_1,W_1; V_2,W_2)$ by 
\begin{equation}\label{eq_dv1v2}
  d(V_1,W_1; V_2,W_2)=\lara{\undd V_1}{q^{-1}(\undd W_2-C_q \undd V_2)}+\lara{\undd V_2}{q\undd W_1}
\end{equation}
and set 
\begin{equation}\label{eq_epsilon}
  \epsilon(V_1,W_1; V_2,W_2)=d(V_1,W_1; V_2,W_2)-d(V_2,W_2; V_1,W_1)
{}\end{equation}
 \subsection{Representation varieties}
In this part, we will explore $\fMb{V}{W}$ as module varieties for a specific algebra. Using the tuples $(B,\alpha,\beta)$, we can define an algebra $\Lambda$ in the following manner. Consider $\Lambda$ to be the algebra correlated with $\Lambda^\bullet(V,W)$. The vertices in the quiver $\widehat{\Gamma}_Q$ of $\Lambda$ consist of $I\times \ZZ$ and the edges in this quiver are determined by
 \begin{align*}
      (h,p):\, (i,p)\to (j,p-1)  &\text{ if $(i,p)\in \widehat{J}$  and } h: i\to j \\ 
      (\bar{h},p):\, (j,p)\to (i,p-1)  &\text{ if $(i,p)\in \widehat{J}$  and } h: i\to j \\ 
      (a,p):\, (i,p)\to (i,p-1)  &\text{ if $(i,p)\in \widehat{I}$ } \\ 
      (b,p):\, (i,p)\to (i,p-1)  &\text{ if $(i,p)\in \widehat{J}$ } 
  \end{align*} 
  Let $R$ represent the idea generated by the equations \refb{eq_formularelation}. Thus, we can express $\Lambda$ as $k\widehat{\Gamma}_Q/R$. In this context, we can interpret $\Lambda$ as the algebra $\mathcal{R}$ as defined in \cite{KS}. \\

\noindent
Following \cite{CFR2}, if $W$ subjects to \refb{eq_wip}, then the algebra $\Lambda$ originates from the category $\cH_Q$. Here, the objects in $\cH_Q$ consist of exact sequences
\[0\to P\to Q\]
where $P,Q$ are the projective representation of $Q$, The morphisms in this category are a tuple of morphisms $(f,g)$, ensuring that the following diagram is commutative
\[
\xymatrix{
  & 0\ar[r] & P\ar[r]\ar[d]^f & Q \ar[d]^g\\ 
  & 0\ar[r] & P'\ar[r] & Q'
}\]
for any pair $(P\to Q), (P'\to Q')$. The category is generated by projective covers of nonprojective indecomposable representations $M_\beta$ of $Q$ in addition to projective representations. That is, we take the projective cover of the indecomposable representation $M_\beta$.
\[0\to P_\beta \to Q_\beta \to M_\beta\to 0\]
if $M_\beta$ is not a projective representation of $Q$ and $(0\to P_\beta \to Q_\beta)=(M_\beta=M_\beta)$ if $M_\beta$ is a projective representation. For simplicity, we also write $(P_\alpha\to Q_\alpha)$ for $(M_\beta=M_\beta)$ if $M_\beta$ is a projective representation of $Q$.  Set 
\begin{equation}
  B_Q:=\End_{\cH_Q}(\bigoplus_{\beta\in R^+}(P_\beta\to Q_\beta))^{op}
\end{equation}
Following \cite[Theorem 4.11]{CFR2}, we have 
\begin{equation}\label{eq_BQ}
  B_Q\cong \Lambda_{\mid \Gamma_Q} 
\end{equation}
The $\Res $ functor will be introduced, mapping from $\Lambda-\omod$ to $kQ-\omod$. Let $M$ be a representation of $\Lambda_{\mid \Gamma_Q}$, we write it for $(V,W,x)$, where $V,W$ are $\Gamma_Q$-graded vector spaces such that $V_\beta \neq 0 $ only if $M_\beta$ is not a projective representation, and $W_\beta\neq 0$ only if $M_\beta$ is a projective representation. One defines 
\begin{equation}\label{eq_Reslam}
  \Res:\,  \Lambda_{\mid \Gamma_Q}-\omod \to kQ-\omod
\end{equation}
by sending $\cM=(V,W,x)$ to $(W,y)$ where $y:W_{i}\to W_{j}$ is given by the composition $x_{h_1}x_{h_2}\cdots x_{h_r}$ of a following path in the Auslander-Reiten Quiver $\Gamma_Q$ of $Q$.
\begin{equation}\label{eq_resmaps}
  P_{i}\stackrel{h_r}{\rightarrow} S_i\stackrel{h_{r-1}}{\rightarrow}\cdots \stackrel{h_2}{\rightarrow}\tau^{-1}S_j\stackrel{h_1}{\rightarrow}P_j
\end{equation}
where $\tau$ is the Auslander-Reiten translation. \\ 

\noindent
The functor $\widehat{-}$ will be introduced, mapping from $kQ-\omod$ to $\Lambda-\omod$. Consider $M$ as a representation of $Q$, and for every positive root $\beta\in R^+$, we construct 
\begin{equation}
  \widehat{M}_\beta:= \Ima(\homq{Q_\beta}{M}\to \homq{P_\beta}{M})
\end{equation}
and the matrices $x_h: \widehat{M}_\beta\to \widehat{M}_{\beta'}$ by 
\begin{equation}
  \xymatrix{
  &\homq{Q_\beta}{M}\ar[d]_{\homq{f}{M}}\ar@{->>}[r]&\widehat{M}_\beta\ar@{-->}[d]^{x_h} \ar@{^{(}->}[r] &\homq{P_\beta}{M}\ar[d]^{\homq{g}{M}}\\ 
  &\homq{Q_{\beta'}}{M}\ar@{->>}[r]&\widehat{M}_{\beta'} \ar@{^{(}->}[r] &\homq{P_{\beta'}}{M}
  }
\end{equation}
where $f,g$ are induced by $M_{\beta'}\to M_\beta$. Therefore, we obtain the functor 
\begin{equation}\label{eq_widehat}
\begin{split}
    \widehat{-}: kQ-\omod&\to \Lambda-\omod\\ 
       M &\mapsto \widehat{M}
\end{split}
\end{equation}
In the following discussion, we will denote $i$ as the dimension vector of indecomposable projective representation $P_i$, and we will consistently refer to $\beta$ as the roots of non-projective indecomposable representations of $Q$. Given a dimension vector $\mathbf{d}$ for $\Lambda_{|\Gamma_Q}$, we can represent it as $$\mathbf{d} = \mathbf{d}_V + \mathbf{d}_W,$$ where $\mathbf{d}_V$ has support in $R^+\setminus\{i\}_{i\in I}$ and $\mathbf{d}_W$ has support in $\{i\}_{i\in I}$. For the corresponding graded vector $(V,W)$, we can view $\Lambda^\bullet(V,W)$, as shown in equation \refb{eq_formularelation}, as the representation variety $\Rep_\mathbf{d}(\Lambda)$. If $\mathbf{d}_W = \gamma$, then Theorem \ref{theo_desingular} implies
\begin{equation}
  \Phi: \fMba{W}\cong E_\gamma
 \end{equation} 

\noindent
To define the graded quiver varieties using an algebraic approach, the concept of stable points in $\Rep_{\bd}(\Lambda)$ must be explained. A module $M$ over $\Lambda$ is considered stable if it satisfies $\Hom_{\Lambda}(S_{\alpha}, M) = 0$ for all simple modules $S_{\alpha}$ supported on nonprojective roots. On the other hand, a module $M$ is called costable if it satisfies $\Hom_{\Lambda}(M, S_{\alpha}) = 0$ for all simple modules $S_{\alpha}$ supported on nonprojective roots. The subvariety $\Lambda^{st}(V, W)$ consists of stable modules in $\Rep_{\bd}(\Lambda)$. It is observed that stable modules align with stable points as defined in \cite{Na}, where any submodule $S_{\alpha}$ of $M$ remains stable under the matrices $B_h$ and $\beta(S_{\alpha}) = 0$ for equation $\ref{eq_Bab}$.\\ 

\noindent
Following \cite[Section 3.4]{KS1}, we define \emph{graded quiver variety} by the $G_V$-quotient
\begin{equation}
    \fMb{V}{W}:= \Lambda^{st}(V,W)/G_V
\end{equation}
Given an element $g$ belonging to the group $G_V$, the equation $\Res g\cdot\cM=\Res \cM$ holds for any module $\cM$. In this case, the functor $\Res$ described in equation \refb{eq_Reslam} yields a proper mapping.
\begin{equation}\label{eq_pifMb}
     \pi : \fMb{V}{W}\to \fMba{W}=E_\gamma
 \end{equation}
 In \cite[Section 3.4]{KS1}, we have 
 \begin{equation}
  \fMbr{V,W}\cong \{[\cM]\in \fMb{V}{W}\mid\, \cM \text{ is a costable module.}\}
 \end{equation}
 Let $M$ be an element of $\fMba{W}$. It can be observed that the module $\widehat{M}$ is both stable and costable. This implies the existence of a distinct graded vector space $V$ such that the class $[\widehat{M}]$ belongs to $\fMb{V}{W}$. As mentioned in the reference \cite[Section 3.2]{LP}, the module $\widehat{M}$ acts as a functor on the split Grothendieck group $K_0^{split}(kQ)$ which is constructed based on the isomorphic classes of representations of $Q$, subject to the given relation $[M]+[N]-[M\oplus N]=0$. It is demonstrated that the module $\widehat{M}$ is uniquely determined by the representation $M$ up to isomorphism. This mapping is crucial in understanding the relationships between different modules in this context. We see that the following map
\begin{equation}
\begin{split}
    \undd \widehat{-}: K_0^{split}(kQ)&\to \ZZ[\beta]_{\beta\in R^+}\\ 
                      M&\mapsto \undd \widehat{M}
\end{split}
\end{equation}
is bijective.
 \begin{Definition}
Let $\NN[R^+]$ represent the monoid formed by the positive roots $R^+$. If we have two elements $v$ and $v'\in \NN[R^+]$, we define $v\geq v'$ as $v_{\beta}\geq v'_{\beta}$ for all $\beta\in R^+$. Assuming that $\bd_W=\gamma$ for a specific $\gamma\in \Qp$, we can define $V(\gamma)$ as the collection of elements $v\in \NN[R^+]$ such that $\undd \widehat{M}=v+\gamma$ for a given $M\in E_{\gamma}$. According to Theorem \ref{theo_desingular}, there exists a one-to-one correspondence $\Pi$ described by
   \begin{equation}\label{eq_Pi}
   \begin{split}
             \Pi: \KP{\gamma} &\to V(\gamma)\\ 
          \lambda &\mapsto \undd \widehat{M_\lambda}-\gamma
   \end{split}
   \end{equation}
   Let's consider the case where $\lambda' \leq \lambda$, indicating that $\OO_\lambda\subset \overline{\OO}_{\lambda'}$. By applying Theorem \ref{theo_desingular}, we can observe that 
   $$\fMbr{\Pi(\lambda),W}\subset \overline{\fMbr{\Pi(\lambda'),W}}$$ 
   This implies, based on equation \ref{eq_strareg}, that $\Pi(\lambda)\leq \Pi(\lambda')$.Therefore, we can conclude $\Pi(\lambda')\geq \Pi(\lambda)$ for any $\lambda'\leq\lambda$ in $\KP{\gamma}$. It is evident from Theorem \ref{theo_desingular} that the set $V(\gamma)$ consists of dimension vectors $v \in \mathbb{N}\Gamma_Q$ such that there exists a stable and costable module $M$ over $\Lambda$ with $\undd M = v+\gamma$. For simplicity, let's denote the image of $\Pi(\lambda)$ for $\lambda \in \KP{\gamma}$ as $v_\lambda$. Referring to \cite{CFR1} or \cite[Corollary 3.15]{LP}, we can express $v_\lambda$ as $(v_U)_{U \in \Gamma_Q}$, where
   \begin{equation}\label{eq_vlambda}
   v_U=\odim\Ima (\homq{Q_U}{M}\to \homq{P_U}{M}). 
   \end{equation}
   for the projective resolution of indecomposable non-projective representations $U$, and $v_U=\dim \homq{U}{M}$ if $U$ is a projective representation. It is easy to see $\Pi(\mu\oplus\nu)=\Pi(\mu)\oplus\Pi(\nu)$.
\end{Definition}
 Following \cite[Section 4.5]{KS1}, one can define the $CK$ functor by 
     \begin{align*}
         CK:\, kQ-\omod&\to k\Gamma_Q-\omod\\ 
            M   &\mapsto (\extq{M_\beta}{M})_{\beta}
     \end{align*}
     where $\beta$ runs over the non-projective roots. 
 \begin{Lemma}\label{lem_CK}
  This functor is identified with the definition given by \cite[Section 4.5]{KS1}
 \end{Lemma}
\begin{proof}
For a representation $M$ of $Q$, the functor $K_R$ from \cite{KS1} when restricted to the category $\mathcal{H}_Q$ is defined as $$K_R(M)(\iota: P\to Q)=\homq{Q}{M}\to \homq{P}{M}$$ for any $\iota: P\to Q$. Similarly, the functor $K_{LR}$ from \cite{KS1} when restricted to the category $\mathcal{H}_Q$ is equal to $\widehat{M}$. As described in \cite[Section 4.5]{KS1}, the functor $CK(M)$ can be seen as the cokernel of the natural map $\widehat{M}\to K_R(M)$, specifically given by
         \[CK(M)(\iota: P\to Q)=\extq{\coker \iota}{M}\]
The map $\widehat{M}\to K_R(M)$ can be regarded as the map $$\widehat{M}(\iota: P\to Q)=\Ima \homq{\iota}{-}\to \homq{P}{M}$$ for all $\iota:\, P\to Q$.
    Referring to \cite{CFR}, it is known that the support of $CK(M)$ lies within the set of non-projective indecomposable modules $U$ over $kQ$. Additionally, since $CK(M)(\iota: P_U\to Q_U)=\extq{U}{M}=D\homq{M}{\tau U}$ for the projective resolution of $U$ and $CK(M)(\Id : P\to P)=0$, it follows that $CK(M)$ is a module of the Auslander algebra $k\Gamma_Q$ according to \cite[Section 6.3]{CFR1}.
\end{proof}

 \noindent
We list some useful properties of graded quiver varieties. 
 \begin{Theorem}[{\cite{KS}}]\label{theo_quivervarieties}
   Under the above assumption, we have
   \begin{enumerate}
         \item $\fMbr{V,W}\Seteq{M\in \fMba{W}}{\undd \widehat{M}=(w,v)}$;
         \item \cite[Lemma 4.2]{KS} The fiber of $\pi: \fMb{V}{W}\to \fMba{W}$ at the point $M$ such that $\undd \widehat{M}=(v_0,w)$ is homeomorphic to the quiver Grassmannian 
         \[\Gr_{v-v_0}(CK(M))\]
         Following \cite[Theorem 3.14]{Na1}, the fiber $\pi^{-1}(M)\cong \pi^{\perp -1} (0)$ where $$\pi^\perp : \fMb{V-V^0}{W^\perp}\to \fMba{W^\perp}$$ where $W^\perp=W-C_q V^0$
     \end{enumerate}  
 \end{Theorem}

\subsection{Hall maps and graded quiver varieties}
Let $W$ be a graded vector space with dimension vector $\undd W=\gamma$ and fix a partition $W=W_1\oplus W_2$ such that $\undd W_1=\alpha$ and $\undd W_2=\beta$. Remember the Hall maps \refb{eq:resdiagram}
\begin{equation}
  E_\alpha\times E_\beta\stackrel{\kappa}{\leftarrow} F_{\bdg}\stackrel{\iota}{\rightarrow} E_\gamma
\end{equation}
 and the birational proper map $\pi:\fMb{V_\lambda}{W^\gamma}\to \fMba{V_\lambda,W^\gamma}=\overline{\OO}_\lambda\subset E_\gamma$, where $\undd V_\lambda=\Pi(\lambda)$ as in \ref{eq_Pi}. It induces the following diagram:
 \begin{equation}\label{dia_pihall}
\begin{tikzcd}
\fMb{V_\lambda}{\gamma} \arrow[d, "\pi"] 
 &Y_\bdg^\lambda  \arrow[d, "\pi_\beta"] \arrow[l, "\iota"] \arrow[r, "\kappa"]
 &\bigcup\limits_{v_1+v_2=v_\lambda}\fMb{v_1}{\alpha}\times \fMb{v_2}{\beta} \arrow[d, "\pi\times\pi"]\\
\overline{\OO}_\lambda 
&F_{\bdg}\cap \overline{\OO}_\lambda \arrow[r,"\kappa_\lambda"] \arrow[l,"\iota"]
&E_\alpha\times E_\beta
\end{tikzcd}
\end{equation}
where
 \[Y_\bdg^\lambda:\Seteq{F\in \fMb{V_\lambda}{\gamma}}{\Res F(W_2)\subset W_2 }\]
  Recall that for a representation $F=(B,\alpha,\beta)$ of $B_Q$, $\Res F$ is given by \refb{eq_resmaps}
  \begin{equation}\label{eq_x_h}
  x_h; W_{(i,p)}\to W_{(j,p-k-1)}, \text{ where $x_h=\beta_{(j,p-k)}B_{h_1}\cdots B_{h_m}\alpha_{i,p}$}
   \end{equation}
   If $x_h(W_{(i,p)}^1)\subset W_{(j,p-k-1)}^1$, then there exists $V_{(i,p-k)}^1$ for any $(i,p-k)$ inductively such that $\Ima \alpha_{(i,p)}$ and $\Ima B_{h_k}$ are contained within $V^1$ for all $k$. This leads to the creation of a subspace $(V^1,W^1)$ such that $F(V^1,W^1)\subset (V^1,W^1)$. Therefore, we can conclude
   \begin{equation}\label{eq_Ybdg}
    Y_\bdg^\lambda=\bigcup_{V_2\subset V_\lambda}\Set{[F]\in \fMb{V_\lambda}{\gamma}}{F(V_2,W_2)\subset (V_2,W_2)}
   \end{equation}
   Given that $F$ represents the orbit $G_V$ of the module with respect to $B_Q$, the subspace $V_2$ can be uniquely identified based on its dimension vector $v_2$. Specifically, for every $g$ in $G_V$ and any representation $F=(B, \alpha, \beta)$, the restriction $\Res F = \Res g \cdot F$. The expression for $Res g \cdot F$ as shown in equation \ref{eq_x_h} is equal to $$\beta_{j,p-k}g_{j,p-k}^{-1}g_{j,p-k} B_{h_1}g_{j',p-k+1}^{-1}...g_{i',p+1}B_{h,m}g_{i,p}^{-1}g_{i,p}\alpha_{i,p}=x_h.$$ Therefore, $g \cdot F(V'_2, W_2) \subset (V'_2, W_2)$ for any isomorphic spaces $V_2, V'_2=g V_2$. As the elements $\fMb{V_\lambda}{\gamma}$ represent the orbits of $F$, it is evident that $V_2$ is uniquely determined by its dimension vector $v_2$. In the subsequent discussions, we will refer to the subspace $V_2$ as $v_2$.
   \begin{Lemma}\label{lem_resF}
  We have 
  \begin{equation}
    \Res (F_{\mid (v_2,W_2)})=(\Res F)_{\mid W_2}
  \end{equation}
\end{Lemma}
\begin{proof}
  The functor $\Res F$ can be represented as
  \begin{equation}
    x_h: W_2^i\to W_2^j  \text{ where } x_h=\beta_jB_{h_1}\cdots B_{h_m}\alpha_i
   \end{equation} 
   The condition $F(v_2, W_2)$ being a subset of $(v_2, W_2)$ indicates that the image of the restrictions of $\alpha_i, B_h$ to $W_2$ is within $V_2$. This implies that the map $x_h$ doesn't change under the restriction $(V_2,W_2)$.
\end{proof}
   \begin{Definition}\label{def_KQ}
  Suppose we have a dimension vector denoted as $v$. A partition of $v$, represented as $v_1+v_2=v$, is called a \emph{maximal partition} if there exists an element $F\in Y^{\lambda}_{\bdg}$ such that $F(V_2,W_2)\subset (V_2,W_2)$ and there is no other $V'_2$ such that $V_2\subset V'_2\neq V_2$ and $F(V'_2,W_2)\subset (V'_2,W_2)$. In simpler terms, $v_2$ is the largest dimension vector of the subspaces $V'_2$ that satisfy $F(V'_2,W_2)\subset (V'_2,W_2)$ for an element $F\in Y^{\lambda}_{\bdg}$. The set of all maximal partitions of $v$ is denoted as $\KQ{v; w_1,w_2}$.
   \end{Definition}
   \noindent
    Considering \refb{eq_Ybdg}, we obtain the following theorem.
    \begin{Theorem}
        For a partition $\apb=\gamma$ and $\lambda\in \KP{\gamma}$, the variety $Y_\bdg^\lambda$ can be decomposed into a disjoint union:
         \begin{equation}\label{eq_Ybdgdecom}
  Y_\bdg^\lambda=\bigsqcup_{(v_1,v_2)\in\KQ{v; w_1,w_2}} Y_\bdg^\lambda(v_1,v_2)
 \end{equation}
 where each $Y_\bdg^\lambda(v_1,v_2)$ is an irreducible variety. 
    \end{Theorem}
    \begin{proof}
    Let us express $Y_\bdg^\lambda$ in the following manner: \[ Y_\bdg^\lambda = \bigsqcup_{(v_1,v_2)\in\KQ{v; w_1,w_2}} Y_\bdg^\lambda(v_1,v_2) \]
Here, $Y_\bdg^\lambda(v_1,v_2)$ consists of elements $F$ where $v_2$ represents the largest $F$-stable subspace of $V$. It is important to note that for any distinct partitions $(v_1,v_2)$ and $(v'_1,v'_2)$, we have $$Y_\bdg^\lambda(v_1,v_2)\cap Y_\bdg^\lambda(v'_1,v'_2)=\emptyset.$$ If we assume the existence of a non-zero element $F$ in $Y_\bdg^\lambda(v_1,v_2)\cap Y_\bdg^\lambda(v'_1,v'_2)$, then it implies $F(V_2, W_2)\subset (V_2, W_2)$ and $F(V'_2, W_2)\subset (V'_2, W_2)$. Considering $V'=V_2+V'_2$, it becomes evident that $F(V_2+V'_2, W_2)\subset (V_2+V'_2, W_2)$. This contradicts the assumption that $v_2$ is maximal.\\ 
   
\noindent
   Let $\kappa_{v_1,v_2}$ be the map 
   \begin{equation}\label{eq_kappav12}
   \begin{split}
    \kappa_{v_1,v_2}:  Y_\bdg^\lambda(v_1,v_2)&\to \fMb{V_1}{\alpha}\times \fMb{V_2}{\beta}\\  
             F &\mapsto F_{\mid (V_1, W_1)} \times F_{\mid (V_2,W_2)}
   \end{split}
   \end{equation}
where $V_1= V/V_2$.
   We aim to demonstrate the stability of $F_{\mid (V_2,W_2)}$. In simpler terms, we seek to show that no nontrivial subspace $V' \neq 0$ can exist within $V_2$ such that $F_{\mid (V', 0)}$ is a subrepresentation of $F_{\mid (V_2,W_2)}$. If such a subspace $V'$ were to exist, it would contradict the stable nature of $F$ since $F_{\mid (V', 0)}$ would then be a non-zero subrepresentation of $F$. Furthermore, we will establish the stability of the representation $F_{\mid (V_1,W_1)}$. If there were a nonzero subspace $V'$ such that the restriction of $F_{\mid (V_1,W_1)}$ to $(V',0)$ is a subrepresentation of $F_{\mid (V_1,W_1)}$, then the combined subspace $(V_2 \oplus V', W_2)$ would satisfy $F(V_2 \oplus V', W_2) \subset (V_2 \oplus V', W_2)$. However, this scenario is impossible for $(V_2, W_2)$ given our assumption that $V_2$ is maximal. Consequently, we can conclude that the map $\kappa_{v_1,v_2}$ is indeed well-defined. \\
   
\noindent
Subsequently, it will be demonstrated that $Y_\bdg^\lambda(v_1,v_2)$ is an irreducible component for each maximal partition $(v_1,v_2)$. According to the reference \cite{VV1}, a cocharacter $\chi=q\Id_{W_2}\oplus \Id_{W_1}$ is established. As mentioned in \cite[Remark 3.4]{VV1}, it is observed that 
 \begin{equation}\label{eq_kappav1v2}
  \kappa_{v_1,v_2}: Y_\bdg^\lambda(v_1,v_2)\to \fMb{v_1}{\alpha}\times \fMb{v_2}{\beta}
 \end{equation}
 is a vector bundle with fibers of dimension $d(v_1,\alpha,v_2,\beta)$ (refer to equation \ref{eq_dv1v2}). Consequently, it can be concluded that $Y_\bdg^\lambda(v_1,v_2)$ is an irreducible variety.  This implies that equation \refb{eq_Ybdgdecom} gives rise to a set of irreducible components of $Y_\bdg^\lambda$.
    \end{proof}

\subsection{Open subvariety}
Let us assume that $\lambda\in\ext(\adb)$.
  We see that the restriction $\pi_\beta$ to $\overline{\OO}_\lambda\cap F_\bdg$ is birational, and its restriction to the open dense subvariety $\iota^{-1}\fMbr{V_\lambda,\gamma}$ of $Y_\bdg^\lambda$ is isomorphic to $\OO_\lambda\cap F_\bdg$. 
The Corollary \ref{cor_irrOOlam} implies that
\begin{equation}\label{eq_irrOO2}
    \iota^{-1}\fMbr{V_\lambda,\gamma}\cong \OO_\lambda\cap F_\bdg=\bigcup_{(\munu)\in \ext_\adb^{ger}(\lambda)}U(\munu)
\end{equation}
 where $U(\munu)$ is the irreducible components associated with $(\munu)$.

 This means that the set of components of $Y_\bdg^\lambda$ coincides with the set of components of $\Gr_{\beta}(M_\lambda)$.  It implies that for each $(\munu)\in\ext_\adb^{ger}(\lambda)$, there exists a unique maximal pair $(v_1,v_2)$ such that
 \begin{equation}\label{eq_UmunuY}
    Y_\bdg^\lambda(v_1,v_2)=\overline{U(\munu)}^{Y}
  \end{equation} 
  Here $\overline{U(\munu)}^{Y}$ refers to the closure of the inverse image of $U(\munu)$ under the map $\pi_\beta$ in the $Y_\bdg^\lambda$. In other words, we have the following lemma.
  \begin{Lemma}
  For a partition $\apb=\gamma$ and $\lambda\in \KP{\gamma}$, we have
        \begin{equation}\label{eq_extKP}
    \ext_\adb^{ger}(\lambda)=\KQ{v_\lambda;\adb}
  \end{equation}
  \end{Lemma}

\noindent
Let us review the diagram \refb{dia_pihall}, we will obtain the following diagram. 
\begin{Lemma}
    For a maximal partition $v_1+v_2=v_\lambda$, we will show that the following diagram is commutative,
\begin{equation}\label{dia_v1fmb}
\begin{tikzcd}
Y_\bdg^\lambda(v_1,v_2) \arrow[d, "\pi_\beta"] \arrow[r, "\kappa_{v_1,v_2}"]
 &\fMb{v_2}{\alpha}\times \fMb{v_1}{\beta}  \arrow[d, "\pi\times\pi"] \\
\overline{U(\munu)} \arrow[r,"\kappa_\lambda"]
&\overline{\OO_\mu}\times\overline{\OO_\nu}
\end{tikzcd}
\end{equation}
 where $(\munu)$ is the pair corresponding to $(v_1,v_2)$ via \refb{eq_UmunuY}.
\end{Lemma}
\begin{proof}
     Taking $F\in Y_\bdg^\lambda(v_1,v_2)$, we have 
  \begin{align*}
    \kappa_\lambda\pi_\beta(F)&=\kappa_\lambda \Res F\\ 
    &=((\Res F)_{\mid W_2},(\Res F)_{\mid W_1})\\ 
    &=(\Res F_{\mid V_2,W_2}, \Res F_{\mid V_1,W_1}) \quad \text{ By Lemma \ref{lem_resF}}\\ 
    &=(\pi\times\pi) \kappa_{v_1,v_2}(F)
  \end{align*}
where $\pi \times \pi: \bigsqcup_{v_1+v_2=v}\fMb{v_1} {\alpha}\times\fMb{v_2}{\beta}\to \bigcup_{\munu\in\ext_{\adb}^{\mins}(\lambda)}\overline{\OO}_\mu\times \overline{\OO}_\nu$ is defined by 
 \begin{equation}
  (F_{\mid(V_1,W_1)},F_{\mid(V_2,W_2)})\mapsto (\Res F_{\mid(V_1,W_1)},\Res F_{\mid(V_2,W_2)})
 \end{equation}
\end{proof}

 \begin{Remark}\label{rem_Ybdg}
Note that the map $\pi\times\pi$ does not always result in a birational map, as $$\widehat{\Res F}\neq F$$ in most cases (see \cite[Theorem 3.4]{CFR2}). Specifically, we have $\odim \widehat{\Res F}\leq \odim F$, with the equality holding only when $F$ is a bistable point. Therefore, the map $\pi\times\pi$ is birational if and only if $(v_\mu,v_\nu)=(v_1,v_2)$ which implies $M_\lambda=M_\nu\oplus M_\mu$.

It is worth noting that the decomposition \refb{eq_Ybdg} of $Y_\bdg^\lambda$ coincides with the decomposition in the sense of Varagnolo and Vasserot \cite{VV1} or Nakajima \cite[Section 3.5]{Na1}. Remember that the following diagram for $E_\gamma$ can be found in \cite[Section 3.1]{VV1}:
\begin{equation}
\begin{tikzcd}
\bigcup_{v_1+v_2=v}\fMb{v_1}{\alpha}\times \fMb{v_2}{\beta} \arrow[d, "\pi\times\pi"]
& \iota^{-1}(\fMb{v}{\gamma}) \arrow[l,"\kappa'"]\arrow[r, "\iota'"]  \arrow[d, "\pi_\beta"] 
& \fMb{v}{\gamma} \arrow[d, "\pi"]\\
\bigcup_{v_1+v_2=v}\fMba{v_1,\alpha}\times \fMba{v_2,\beta} 
&F_{\bdg}\cap \fMba{v,\gamma} \arrow[l, "\kappa" ] \arrow[r,"\iota"]
&  \fMba{v,\gamma}
\end{tikzcd}
\end{equation}
Here 
\begin{equation}
  \iota^{-1}(\fMb{v}{\gamma})=\bigcup_{v_1+v_2=v} \mathfrak{Z}(v_1,\alpha,v_2,\beta) 
\end{equation}
such that $\mathfrak{Z}(v_1,\alpha,v_2,\beta)$ is a vector bundle over $\fMb{V_1}{\alpha}\times \fMb{V_2}{\beta}$ of rank $d(v_1,\alpha,v_2,\beta)$ (see \refb{eq_dv1v2}). That means $\mathfrak{Z}(v_1,\alpha,v_2,\beta)=Y_\beta(v_1,v_2)$.
 \end{Remark}

 \subsection{A birational map}
Let's consider the extension variety denoted by $\cExt{\lambda}{\mu}{\nu}:=\kappa_\lambda^{-1}(M_\mu,M_\nu)$. We are going to define a birational map on this extension variety. According to Lemma \ref{lem:dimfiberkappa}, for a pair $(\mu,\nu)$, it is stated that the subset $ \kappa_\lambda^{-1}(M_\mu,M_\nu) \cap \OO_\lambda$ is an open subset of $ \kappa_\lambda^{-1}(M_\mu,M_\nu)$. Let's remember the decomposition given in equation \refb{eq_irrOO2}.
\begin{equation}\label{eq_sqcupOO}
    \kappa_\lambda^{-1}(M_\mu,M_\nu)\cap \OO_\lambda\subset \OO_\lambda\cap F_\bdg=\bigsqcup_{(\mu',\nu')\in\ext_\adb^{ger}(\lambda)}U(\mu',\nu')
\end{equation}
The disjoint union follows from \refb{eq_Ybdgdecom} and \refb{eq_UmunuY}. This means that for each irreducible component $X_i$ of $ \cExt{\lambda}{\mu}{\nu}$ there exists a unique $(\mu',\nu')\in \ext_\adb^{ger}(\lambda)$ such that 
\[X_i \subset \overline{U(\mu',\nu')}\]
. 

We will give a birational map for each irreducible component $X_i$ of $\cExt{\lambda}{\mu}{\nu}$.
It is enough to consider the diagram \refb{dia_v1fmb} for $(\mu',\nu')$ to know $X_i$. 
\begin{equation}\label{dia_Yv1v2}
\begin{tikzcd}
Y_\bdg^\lambda(v_1,v_2) \arrow[d, "\pi_\beta"] \arrow[r, "\kappa_{v_1,v_2}"]
&\fMb{v_2}{\alpha}\times \fMb{v_1}{\beta} \arrow[d, "\pi\times \pi"]\\
\overline{U(\mu',\nu')} \arrow[r, "\kappa_\lambda"]
 &\overline{\OO_{\mu'}}\times\overline{\OO_{\nu'}}
\end{tikzcd}
\end{equation}
  where $(v_1,v_2)$ is given by \refb{eq_UmunuY}. Let us define $\pi\cExt{\lambda}{\mu}{\nu}_i$ by
\begin{equation}
  \pi\cExt{\lambda}{\mu}{\nu}_i=\kappa_{v_1,v_2}^{-1}(\Gr_{v_{1}-v_{\mu}}(CK(M_\mu))\times \Gr_{v_{2}-v_{\nu}}(CK(M_\nu)))
\end{equation}
where $\kappa_{v_1,v_2}$ is the map in \refb{dia_Yv1v2}, and $\Gr_{v_{1}-v_{\mu}}(CK(M_\mu))$ and $\Gr_{v_{2}-v_{\nu}}(CK(M_\nu))$ are quiver Grassmannians of $CK(M_\mu)$ and $CK(M_\nu)$, respectively. By Theorem \ref{theo_quivervarieties}, we have
\[(\pi\times\pi)^{-1}(M_\mu,M_\nu)=\Gr_{v_{1}-v_{\mu}}(CK(M_\mu))\times \Gr_{v_{2}-v_{\nu}}(CK(M_\nu))\]
It follows that 
$$\pi\cExt{\lambda}{\mu}{\nu}_i=\kappa_{v_1,v_2}^{-1}(\pi\times\pi)^{-1}(M_\mu,M_\nu)=\pi_\beta^{-1}(X_i)$$
Since the restriction $\pi_\beta$ to $\OO_\lambda\cap \kappa_\lambda^{-1}(M_\mu,M_\nu)$ is an isomorphism, we obtain the following birational map
\begin{equation}
  \pi_\beta: \pi\cExt{\lambda}{\mu}{\nu}=\bigcup_{i}\pi\cExt{\lambda}{\mu}{\nu}_i\to \cExt{\lambda}{\mu}{\nu}=\bigcup_i X_i
\end{equation}\\ 

\noindent
On the other hand, $\pi\cExt{\lambda}{\mu}{\nu}$ is a vector bundle over $$\Gr_{v_{1}-v_{\mu}}(CK(M_\mu))\times \Gr_{v_{2}-v_{\nu}}(CK(M_\nu))$$ with dimension of the fibers $d(v_1,\alpha,v_2,\beta)$ by \refb{eq_kappav1v2}. Therefore, we have the following Lemma.

\begin{Lemma}\label{lem_extbiration}
  Given a pair $(\munu)\in\KP{\alpha}\times\KP{\beta}$, let $\lambda\in \ext(\munu)$. There is a unique maximal pair $(v_1,v_2)\in \KQ{v_\lambda}$ such that 
  \[\pi_\beta: \pi\cExt{\lambda}{\mu}{\nu}\to \cExt{\lambda}{\mu}{\nu}\] 
  is a birational map. Here $\pi\cExt{\lambda}{\mu}{\nu}$ is a vector bundle over $$\Gr_{v_{1}-v_{\mu}}(CK(M_\mu))\times \Gr_{v_{2}-v_{\nu}}(CK(M_\nu))$$ with dimension of the fibers $d(v_1,\alpha,v_2,\beta)$.
 \end{Lemma}

\section{Restriction and Induction functors}
In this section, we review the Lusztig's category and the perverse sheaves on graded quiver varieties. 
\subsection{Lusztig's sheaves}\label{sec_lus}
Let $\alpha\in\Qp$ and $\la I\ra_\alpha$ be the set of words $\squw{i}{r}$ of $I$ such that $\sum_{k=1}^r\alpha_{i_k}=\alpha$. For a word $\bfi=\squw{i}{r}\in \wI{\alpha}$ and a $I$-graded vector space $V$ such that $\undd V=\alpha$, let us define its flag variety as follows.
\begin{equation}\label{eq:flagwords}
  \cF_\bfi:=\{(0=V_0\subset V_1\subset\cdots \subset V_{r-1}\subset V_{r}=V )\mid\, \text{ such that $\odim V_{k}/V_{k-1}=\alpha_{i_k}$}\}
\end{equation}
There exists a vector bundle over $\cF_\bfi$, which is defined as follows:
\begin{equation}\label{eq:repQflag}
  \Fit:=\{(\phi,x)\mid\, \text{ $\phi\in \cF_\bfi$ and $x\in \RepQ$ such that $x(V_k)\subset V_k$}\}
\end{equation}
It yields a proper map:
\begin{equation}\label{eq:proper map}
\begin{split}
p:\, \Fit&\to \RepQ\\
(\phi,x)&\mapsto x
\end{split}
\end{equation}

Let's revisit the perverse sheaves on the representation variety $\RepQ$. As both $\Fit$ and $\RepQ$ are smooth spaces, there exist constant perverse sheaves on them: $\II_{\Fit}=\CC_{\Fit}[\odim \Fit]$ and $\II_{\RepQ}=\CC_{\RepQ}[\odim \RepQ]$. The map $p$ induces a complex $\cL_{\bfi}=p_!\II_{\Fit}[h_{\bfi}]$ for a parameter $h_{\bfi}$, known as Lusztig's sheaf. For simplicity, we will not delve into the specifics of $h_{\bfi}$ in this discussion. According to the BBD Decomposition theorem, Lusztig's sheaves can be expressed as a direct sum of simple perverse sheaves on $\RepQ$. Let $\cP_{\alpha}$ denote the collection of simple perverse sheaves present in a component of $\cF_{\bfi}$ for all word $\bfi\in\wI{\alpha}$. We will denote this collection as $\cP$ when there is no risk of confusion. Since the simple perverse sheaves can be represented as $\IC(X,\lambda)$ where $X$ denotes locally closed smooth subspaces of $\RepQ$, $\lambda$ represents a local system on $X$, and $\IC(X,\lambda)$ is the simple perverse sheaf corresponding to them. To simplify matters, we will refer to $\lambda$ or $\IC(\lambda)$ instead of $\IC(X,\lambda)$. In a word, we get
\begin{equation}\label{eq:BBDlusztig}
  \cL_\bfi=\bigoplus_{\lambda\in \cP}L_\bfi(\lambda)\boxtimes \IC(\lambda)
\end{equation}
where $L_\bfi(\lambda)$ are $\ZZ$-graded vector spaces. We denote by $\cQ_\alpha$ the semisimple full subcategory of $D_{G_\alpha}^b(E_\alpha)$ generated by $\cP_\alpha$.\\

\noindent
In particular, in Dynkin cases, all the simple perverse sheaves are of the form $\IC{(\CC_{\OO}[\odim \OO])}$ where $\OO$ run over all $G_\alpha$-orbits.
\subsection{Borel-Moore Homology theory}
In this section, we will recall the notion of Borel-Moore homology and assume that $X$ is equidimensional.  Denote by $d$ the dimension of $X$.
\begin{Definition}\label{def:BMhomlogy}
For an algebraic variety $X$ over $\CC$, let $\DD_X$ be the dual complex of $X$, one defines its \emph{Borel-Moore Homology} by
\[\HH_k(X):= \HH^{-k}(X,\DD_X)=\HH^{-k}(p_*\DD_X)\]
where $p:X\to \Spec(\CC)$. We remark that $ \HH^{-k}(p_*\DD_X)\cong \HH^k(p_!\CC_X)$, as $\DD(p_!\CC_X)=p_*\DD_X$. This means $\HH_k(X)=\HH^k(p_!\CC_X)$.
\end{Definition}
For $k=2d$ we have that the components of $X$ form a basis of $\HH_{2d}(X)$. If $i:Y\subset X$ is a closed subvariety of $X$ with dimension $r$, then the map $i_*i^!\DD_X\to \DD_X$ and equation $i^!\DD_X=\DD_Y$ induce a map $i_*:\,\HH_k(Y)\to \HH_k(X)$. It leads to $i_*[Y]\in\HH_{2r}(X)$. Thus, we can consider $[Y]$ as an element in $\HH_{2r}(X)$.

\begin{Definition}\label{def_poly}
For a variety $X$ over $\CC$, we define its \emph{Poincare polynomial} by
\[\Po{X}:= \sum_{k\in \NN}\odim\HH_k(X)q^k\]
\end{Definition}
\begin{Lemma}\label{lem_fiberhom}
  Let $\pi : X\to Y$ be a proper map and $\CC_X$ be the constant sheaf over $X$. For the complex $\pi_!\CC_X$ and any point $y\in Y$, we have 
  \begin{equation}
    (\pi_!\CC_X)_y\cong \HH_\bullet(\pi^{-1}(y))
  \end{equation}
 If for every point $y$ in $Y$, the fiber of $\pi$ is isomorphic, then we obtain
  \[\pi_!\CC_X\cong \HH_\bullet(\pi^{-1}(y))\CC_Y\]
\end{Lemma}

\subsection{Restriction Functor}
Let us recall the maps in \refb{eq:resdiagram}, we define the restriction functor as
\begin{equation}\label{eq:Resfunctor}
\begin{split}
\Res_{\adb}^\gamma: D_{G_\gamma}^b(E_\gamma(Q))&\to D_{G_\alpha\times G_\beta}^b(\RepQ\times E_\beta(Q))\\
\PP&\mapsto \kappa_!\iota^*(\PP)[-\lara{\alpha}{\beta}]
\end{split}
\end{equation}
Let $\lambda\in\KP{\gamma}$ be a Kostant partition of $\gamma$, and $\IClam$ be the simple perverse sheaf associated with $\lambda$. It is well known that
\begin{equation}\label{eq:ResP}
\Res_{\adb}^\gamma(\IC(\lambda))=\bigoplus_{(\mu,\nu)\in\ext_{\adb}(\lambda)}K_\lambda(\munu)\ICmu\boxtimes\ICnu
\end{equation}
where $K_\lambda(\munu)\in \NN[q,q^{-1}]$.\\

\noindent
To investigate the coefficients $K_\Lambda(\munu)$, it is necessary to consider the lemma provided below.
\begin{Lemma}\label{lem:dim}
Let $\gamma=\apb$ be a partition of $\gamma\in \Qp$. Given $\lambda\in \KP{\gamma}$, $\mu\in\KP{\alpha}$, and $\nu\in\KP{\beta}$, we set 
$$d_\lambda(\munu):=\odim \OO_\lambda-\lara{\alpha}{\beta}-\odim \OO_\mu-\odim \OO_\nu$$
then we have
\[d_\lambda(\munu)=\alpha\cdot \beta+\beta\cdot\alpha+[M_\mu,M_\mu]+[M_\nu,M_\nu]-[M_\lambda,M_\lambda]-\lara{\alpha}{\beta}\]
\end{Lemma}
\begin{proof}
To begin with, let us consider the dimension of $\OO_{v}$, which is equal to $\odim\OO_{v}=\odim E_{\eta}-[M_{v}, M_{v}]^1$ where $v$ refers to $\lambda,\munu$ and $\eta$ refers to $\gamma,\alpha,\beta$ respectively. Using equation \refb{prop_eulerform},  we can derive the subsequent equation as shown below
\begin{align*}
  \odim\OO_{v}=&\odim E_{\eta}-[M_{v}, M_{v}]^1\\ 
  =&-\lara{\eta}{\eta}+\eta\cdot \eta -[M_{v}, M_{v}]^1\\ 
  =&\eta\cdot \eta-[M_v,M_v]
\end{align*}
where $\eta \cdot \eta=\sum_{i=1}^n\eta_i^2$. Consequently, we can deduce that: 
\begin{align*}
&\odim \OO_\lambda-\odim \OO_\mu-\odim \OO_\nu-\lara{\alpha}{\beta}\\
=&\gamma\cdot \gamma-[M_\lambda,M_\lambda]-\alpha\cdot \alpha+[M_\mu,M_\mu]-\beta\cdot \beta+[M_\nu,M_\nu]-\lara{\alpha}{\beta}\\ 
=&\alpha\cdot \beta+\beta\cdot\alpha+[M_\mu,M_\mu]+[M_\nu,M_\nu]-[M_\lambda,M_\lambda]-\lara{\alpha}{\beta}
\end{align*}
This leads to the conclusion that  $d_\lambda(\munu)=\alpha\cdot \beta+\beta\cdot\alpha+[M_\mu,M_\mu]+[M_\nu,M_\nu]-[M_\lambda,M_\lambda]-\lara{\alpha}{\beta}$.
\end{proof}

\subsection{Induction functor}\label{sec_IndFun}
Let's revisit the concept of the Induction functor. Within the diagram shown in \refb{dia_induction}, the pull-back functor of $r$ results in a natural equivalence between categories
\[r^*:\, D_{G_\gamma}^b(E_{\bdg})\stackrel{\sim}{\rightarrow}D_{G_\gamma\times G_\alpha\times G_\beta}^b(E_{\beta,\gamma}^{(1)})\]
which has an inverse
\[r_b:\,D_{G_\gamma\times G_\alpha\times G_\beta}^b(E_{\beta,\gamma}^{(1)})\stackrel{\sim}{\rightarrow} D_{G_\gamma}^b(E_{\bdg})\]
We set $r_\sharp=r_b[-\odim r]$.
\begin{Definition}\label{def:Inductionfun}
Remember the maps in equation \refb{dia_induction}. We define the \emph{Induction Functor} by
\begin{equation}
\begin{split}
\star: D_{G_\alpha}^b(E_\alpha)\times  D_{G_\beta}^b(E_\beta) &\to  D_{G_\gamma}^b(E_\gamma)\\
    (\cP, \cP')&\mapsto q_!r_\sharp p^*(\cP\boxtimes\cP')[\odim p]
\end{split}
\end{equation}
\end{Definition}
\begin{Theorem}\label{them:Indorbits}
Let $\adb$ be two elements in $\Qp$. For two Kostant partitions $\munu$ in $\KP{\alpha}$ and $\KP{\beta}$, respectively. We have
\begin{equation}
  \IC(\mu)\star \IC(\nu)=\bigoplus_{\lambda\in \KP{\gamma}} J_\lambda(\munu) \IC(\lambda)
\end{equation}
where $J_\lambda(\munu)\neq 0$ only if there exists some $\lambda'\in \ext(\mu,\nu)$ such that $\lambda\geq\lambda'$.
\end{Theorem}
\begin{proof}
By Lemma \ref{lem:oribitsimage}, the support of $  \IC(\mu)\star \IC(\nu)$ is contained in the set
 \[\bigcup_{\lambda\in\ext(\munu)}\overline{\OO}_\lambda\]
 It implies our conclusion. 
\end{proof}

\subsection{Perverse sheaves on graded quiver varieties}
In this section, we will relate the perverse sheaves to the complex $\pi(V,W)$ induced from the graded quiver varieties (see \cite{Na1}). Below are some notable facts that we would like to mention.

\begin{Theorem}\label{theo:pervers}
Consider an $I$-graded vector space $W$ with $\undd W=\gamma$ and the map $\pi$ defined as in equation \eqref{eq_pifMb}. Define $\pi(V,W)$ using the formula \[\pi(v,w):=\pi_*(\CC_{\fMb{V}{W}}[\odim \fMb{V}{W}])\] then we obtain the following result:
  \begin{enumerate}
     \item \cite[Equation 3.15]{Na1} $\pi(v,w)$ can be decomposed into
            \begin{equation}
               \pi(v,w)=\bigoplus_{V'\in V_v(w)} L_W(V,V')\boxtimes \IC_W(V')
            \end{equation}
            Here $V_v(w)\Seteq{v'\leq v}{\fMbr{V',W}\neq \emptyset}$ and the polynomial of $L_W(V,V')$ is $a_{V,V',W}(t)$ in the sense of \cite[Formula 3.15]{Na1}. Here $a_{V,V,W}(t)=1$, $a_{V,V',W}(t)\neq 0$ unless $V'\leq V$, $a_{V,V',W}(t)=a_{V,V',W}(t^{-1})$, and $$a_{V,V',W}(q)=a_{V-V',0,W^\perp}(q)\in q^{\odim \fMb{V-V'}{W^\perp}}\NN[q^{-2}]$$
           \item  \cite[Proposition 4.8]{SS}, \cite{VV1} Suppose that $\undd W=\gamma$ and a decomposition $W=W_1\oplus W_2$ such that $\undd W_1=\alpha$ and $\undd W_2=\beta$. The restrict functor \refb{eq:Resfunctor} on $\pi(v,w)$ is 
           \[\Res_{\adb} \pi(v,\gamma)=\bigoplus_{v_1+v_2=v} q^{\epsilon(v_2,\beta,v_1,\alpha)}\pi(v_1,\alpha)\boxtimes \pi(v_2,\beta)\]
           where $\epsilon(v_2,\beta,v_1,\alpha)$ is given in \refb{eq_epsilon}.
  \end{enumerate}
\end{Theorem}
\begin{Corollary}\label{cor_desingularorbits}
     Under the assumption in Theorem \ref{theo_desingular}, for each $\lambda\in\KP{\gamma}$ we have a \emph{ resolution} of orbit $\overline{\OO}_\lambda$ with respect to the stratification of $G_\gamma$-orbits.  
     \[\pi_\lambda:\, \fMb{V_\lambda}{W}\to \overline{\OO}_\lambda\]
     for the $\widehat{I}$-graded vector space $V_\lambda=\Pi(\lambda)$. In particular, we have 
      \begin{equation}
                \pi_*(\CC_{\fMb{V_\lambda}{W}}[\odim \fMb{V_\lambda}{W}])=\bigoplus_{\lambda'\geq \lambda} f_{V_\lambda,W}(\lambda')\IC(\lambda')
      \end{equation} 
      where $f_{V,W}(\lambda')\in\NN[q,q^{-1}]$.
 \end{Corollary}
 \begin{proof}
     Following the above theorem and Theorem \ref{theo_desingular}.
 \end{proof}

\begin{Theorem}\label{pro_topdegree}
Considering a partition $\apb=\gamma$ of a dimension vector $\gamma$, let $(\munu)\in\KP{\alpha}\times\KP{\beta}$ and $\lambda\in\ext(\munu)$. Let $K_\lambda(\munu)$ be the coefficients in the decomposition \refb{eq:ResP}, the top degree of $K_\lambda(\munu)$ is less than or equal to $$2e_\lambda(\munu)+d_\lambda(\munu)$$ 

\noindent
If $(\munu)\in\ext_\adb^{ger}(\lambda)$, then 
\[K_\lambda(\munu)=q^{\epsilon(v_1,\alpha,v_2,\beta)}\Po{\Gr_{v_1-v_\mu}(CK(M_\mu))\times \Gr_{v_2-v_\nu}(CK(M_\nu))}\]
where $e_\lambda(\munu)$ is the dimension of $\cExt{\lambda}{\mu}{\nu}$, (see \ref{lem:dimfiberkappa}).
\end{Theorem}
\begin{proof}
Since $\pi: \fMb{V_\lambda}{W}\to \fMba{V_\lambda,W}$ is a proper and birational map, we have
\begin{equation}\label{eq_piic}
  \pi(V_\lambda,W)_{\mid \OO_\lambda}\cong \IClam_{\mid \OO_\lambda}
\end{equation}
Denote by $Z$ the image $\kappa_\lambda(\OO_\lambda\cap F_\bdg)$ and by $j:Z\to E_\alpha\times E_\beta$ its embedding. By Corollary \ref{cor_irrOOlam}, 
we see that $$\overline{Z}=\Ima \kappa_\lambda=\bigcup_{(\munu)\in \ext_\adb^{ger}(\lambda)} \overline{\OO_\mu}\times\overline{\OO_\nu}$$
and 
\[\bigcup_{(\munu)\in \ext_\adb^{ger}(\lambda)} \OO_\mu\times\OO_\nu\subset Z\]
We write $\kappa_Z$ for the restriction of $\kappa_\lambda$ to $\OO_\lambda\cap F_\bdg$. Consider the following diagram
 \begin{equation}\label{dia_Z}
\begin{tikzcd}
\fMb{V_\lambda}{\gamma}\arrow[d, "\pi"]
&\iota^{-1}(\fMb{V_\lambda}{\gamma}) \arrow[l,"\iota'"]\arrow[r, "\kappa"]  \arrow[d, "\pi_\beta"] 
&\bigcup\limits_{v_1+v_2=v_\lambda}\fMb{v_1}{\alpha}\times \fMb{v_2}{\beta} \arrow[d, "\pi\times\pi"]\\
\overline{\OO}_\lambda 
&F_{\bdg}\cap \overline{\OO}_\lambda \arrow[l, "\iota" ] \arrow[r,"\kappa_\lambda"]
&\overline{Z} \\ 
\OO_\lambda \arrow[u,"j'"] 
&F_{\bdg}\cap \OO_\lambda \arrow[u,"j''"] \arrow[l,"\iota''"] \arrow[r, "\kappa_Z"]
&Z \arrow[u,"j"]
\end{tikzcd}
\end{equation}
It follows that
\begin{align*}
  j^*\Res_\adb\pi(V_\lambda,\gamma)&\cong j^*\kappa_{\lambda,!}\iota^*\pi_!\II[-\lara{\alpha}{\beta}] \\ 
     &=\kappa_{Z,!}j''^*\iota^*\pi_!\II[-\lara{\alpha}{\beta}] \qquad &\text{By $j^*\kappa_{\lambda,!}=\kappa_{Z,!}j''^*$}\\ 
     &=\kappa_{Z,!}\iota''^*j'^*\pi_!\II[-\lara{\alpha}{\beta}] \qquad&\text{By $\iota j''=j'\iota''$}\\ 
  &=\kappa_{Z,!}\iota''^* \pi(V_\lambda,\gamma)_{\mid \OO_\lambda} \\ 
  &=\Res_\adb\pi(V_\lambda,\gamma)_{\mid \OO_\lambda}
\end{align*}
Similarly, we see that $j^*\Res_\adb \IClam\cong\Res_\adb\IClam_{\mid \OO_\lambda}$.  The identity \refb{eq_piic} implies that 
\begin{equation}\label{eq_j*Respi}
    j^*\Res_\adb\pi(V_\lambda,\gamma)\cong j^*\Res_\adb \IClam
\end{equation}
By Theorem \ref{theo:pervers}, we obtain
\begin{equation}\label{eq_Respi}
\begin{split}
   j^*\Res_\adb\pi(V_\lambda,\gamma)&\cong \sum_{(v_1,v_2)\in \KQ{v_\lambda;\adb}}q^{\epsilon(v_1,\alpha,v_2,\beta)}j^*(\pi(v_1,\alpha)\boxtimes \pi(v_2,\beta)) \\ 
     &\cong \sum_{(v_1,v_2)\in \KQ{v_\lambda;\adb}}q^{\epsilon(v_1,\alpha,v_2,\beta)}f_{v_1,\alpha}(\mu)f_{v_2,\beta}(\nu)j^*\ICmu\boxtimes\ICnu + j^*P
     \end{split}
\end{equation}
where $(\munu)$ is the unique element of $\ext_\adb^{ger}(\lambda)$ corresponding to $(v_1,v_2)\in \KQ{v_\lambda}$ via equation \refb{eq_extKP}, the support of $P$ is in the subvariety $\overline{Z}/Z$, and $f_{v_1,\alpha}(\mu),\ f_{v_2,\beta}(\nu)\in \ZZ[q,q^{-1}]$.

The equation \refb{eq_j*Respi} implies that 
\begin{equation}
    \Res_\adb \IClam=\sum_{(\munu)\in\ext_\adb^{ger}(\lambda)}q^{\epsilon(v_1,\alpha,v_2,\beta)}f_{v_1,\alpha}(\mu)f_{v_2,\beta}(\nu)(\ICmu\boxtimes \ICnu)+Q
\end{equation}
where the support of $Q$ is in the subvariety $\overline{Z}/Z$.\\ 

\noindent
 We will prove 
\begin{equation}
  f_{v_1,\alpha}(\mu)=\Po{\Gr_{v_1-v_\mu}(CK(M_\mu))}\qquad \emph{and}\qquad f_{v_2,\beta}(\nu)=\Po{\Gr_{v_2-v_\nu}(CK(M_\nu))}
\end{equation}
Let $i_{\mu}$ be the embedding $M_\mu\to E_\alpha$. Since $\pi(v_1,\alpha)=\pi_!\CC[\odim \OO_\mu]$, $\ICmu_{\mid \OO_\mu}=\CC[\odim \OO_\mu]$, and $\Po{\Gr_{v_1-v_\mu}(CK(M_\mu))}=i_{\mu}^*\pi_!\CC$ by Lemma \ref{lem_fiberhom}, then we obtain 
\begin{equation}
  i_{\mu}^*\pi(v_1,\alpha)=\Po{\Gr_{v_1-v_\mu}(CK(M_\mu))}i_{\mu}^*\ICmu
\end{equation}
Considering the equation \refb{eq_Respi}, we see that 
\begin{equation}
  i_{\mu}^*j^*\pi(v_1,\alpha)=f_{v_1,\alpha}(\mu)i_{\mu}^*j^*\ICmu+i_{\mu}^*P'
\end{equation}
where the support of $P'$ is contained in $\overline{\OO}_\mu/\OO_\mu$. Since $M_\mu$ is not in the support of $P'$, then we get $i_{\mu}^*j^*\pi(v_1,\alpha)=f_{v_1,\alpha}(\mu)i_{\mu}^*j^*\ICmu$. The relation $i_{\mu}^*j^*=i_{\mu}^*$ implies $\Po{\Gr_{v_1-v_\mu}(CK(M_\mu))}=f_{v_1,\alpha}(\mu)$. Similarly, we obtain $f_{v_2,\beta}(\nu)=\Po{\Gr_{v_2-v_\nu}(CK(M_\nu))}$ in the same way.

 Therefore, for any $(\munu)\in\ext_\adb^{ger}(\lambda)$, we get 
\begin{align*}
  K_\lambda(\munu)=&q^{\epsilon(v_1,\alpha,v_2,\beta)}f_{v_1,\alpha}(\mu)f_{v_2,\beta}(\nu)\\ 
  =&q^{\epsilon(v_1,\alpha,v_2,\beta)}\Po{\Gr_{v_1-v_\mu}(CK(M_\mu))\times \Gr_{v_2-v_\nu}(CK(M_\nu))}
\end{align*}\\ 

\noindent
For any $(\munu)$ such that $K_\lambda(\munu)\neq 0$, we will show that the top degree of $K_\lambda(\munu)$ is less than or equal to $2e_\lambda(\munu)+d_\lambda(\munu)$. The fact $\pi(v_\lambda,\gamma)=\IClam+I$ for some sheaves $I$ with support on $\overline{\OO}_\lambda/\OO_\lambda$ implies that $K_\lambda(\munu)+f(q)=[\Res\pi(v_\lambda,\gamma):\,\ICmu\boxtimes \ICnu]$ for some polynomial $f(q)\in \NN[q,q^{-1}]$. Let us consider 
\begin{equation}\label{eq_Hlammunu}
  \Res \pi(v_\lambda,\gamma)=\sum_{(\mu',\nu')\leq (\munu)} H_\lambda(\mu',\nu')\IC(\mu')\boxtimes\IC(\nu')+K
\end{equation}
where the support of $K$ doesn't contain $(M_\mu,M_\nu)$. 

Let us consider the following diagram.
 \begin{equation}\label{dia_fibermunu}
\begin{tikzcd}
\fMb{V_\lambda}{\gamma}\arrow[d, "\pi"]
&\iota^{-1}(\fMb{V_\lambda}{\gamma}) \arrow[l,"\iota'"]\arrow[r, "\kappa"]  \arrow[d, "\pi_\beta"] 
&\bigcup\limits_{v_1+v_2=v_\lambda}\fMb{v_1}{\alpha}\times \fMb{v_2}{\beta} \arrow[d, "\pi\times\pi"]\\
\overline{\OO}_\lambda 
&F_{\bdg}\cap \overline{\OO}_\lambda \arrow[l, "\iota" ] \arrow[r,"\kappa_\lambda"]
&\overline{Z}\\ 
\OO_\lambda \arrow[u,"j'"] 
&\cExt{\lambda}{\mu}{\nu} \arrow[u,"i"] \arrow[l,"\iota"] \arrow[r, "\epsilon"]
& (M_\mu,M_\nu) \arrow[u,"i_{\munu}"]
\end{tikzcd}
\end{equation}
It follows that
\begin{align*}
i_{\munu}^*\Res \pi(v_\lambda,\gamma)=&i_{\munu}^*\kappa_{\lambda,!}\iota^*\pi_!\II[-\lara{\alpha}{\beta}]\\ 
=&\epsilon_{!}i^*\iota^*\pi_!\II[-\lara{\alpha}{\beta}]\quad \text{by $i_{\munu}^*\kappa_{\lambda,!}=\epsilon_{!}i^*$}\\ 
=&\epsilon_{!}i^*\pi_!\II[-\lara{\alpha}{\beta}]\quad \text{by $i^*\iota^*=i^*$}  
\end{align*}
Considering the diagram
\begin{equation}
    \xymatrix{
    &\fMb{V_\lambda}{\gamma}\ar[d]^{\pi} &\pi\cExt{\lambda}{\mu}{\nu} \ar[d]_{\pi_\beta}\ar[l]_{i'}  \\ 
     &\overline{\OO}_\lambda &\cExt{\lambda}{\mu}{\nu} \ar[l]_{i} \ar[r]_\epsilon &(M_\mu,M_\nu)
    }
  \end{equation}
  We see that 
\begin{align*}
  \epsilon_{!}i^*\pi_!\II[-\lara{\alpha}{\beta}]=& \epsilon_{!}\pi_{\beta,!}i'^*\II[-\lara{\alpha}{\beta}]\quad \text{by $\pi_{\beta,!}i'^*=i^*\pi_!$}\\ 
  =& \epsilon_{!}\pi_{\beta,!}\CC[\odim \OO_\lambda-\lara{\alpha}{\beta}]\\ 
=&q^{d_\lambda-\lara{\alpha}{\beta}}\HH_\bullet(\pi\cExt{\lambda}{\mu}{\nu}) \qquad\text{by Definition \refb{def:BMhomlogy}}
\end{align*}
Therefore, we obtain $i_{\munu}^*\Res \pi(v_\lambda,\gamma)=q^{d_\lambda-\lara{\alpha}{\beta}}\HH_\bullet(\pi\cExt{\lambda}{\mu}{\nu})$. The formula \refb{eq_Hlammunu} implies that the top degree of $H_\lambda(\munu)i_{\munu}^*\ICmu\boxtimes\ICnu$ is less than or equal to $d_\lambda-\lara{\alpha}{\beta}+2e_\lambda(\munu)$. It follows that the top degree of $H_\lambda(\munu)$ is less than or equal to 
\begin{equation}
  2e_\lambda(\munu)+d_\lambda-\lara{\alpha}{\beta}-d_\mu-d_\nu
\end{equation}%
Since $H_\lambda(\munu)=K_\lambda(\munu)+f(q)$, we obtain the top degree of $K_\lambda(\munu)$ is less than or equal to 
\begin{equation}
  2e_\lambda(\munu)+d_\lambda-\lara{\alpha}{\beta}-d_\mu-d_\nu
\end{equation}%

\end{proof}

\section{Quiver Hecke algebras}
Let us first recall the notion of KLR algebras. Set $m_{i,j}$ as the number of arrows from $i$ to $j$ for $i,j\in I$, and 
 \[
  q_{i,j}(u,v)=\begin{cases}
   0& \text{$i=j$}\\
  (v-u)^{m_{i,j}}(u-v)^{m_{j,i}}& \text{$i\neq j$}
  \end{cases}
 \]
 \begin{Definition}\label{quiver hecke algebra}
 Let $\alpha=\sum_ic_i\alpha_i\in \NN[I]$ and denote $ht(\alpha)=\sum_ic_i=n$. The\emph{ KLR algebra} (quiver Hecke algebra) $R_\alpha$ is an associative $\ZZ[q^\pm]$-algebra generated by \[\{ e_{\bfi}|\bfi\in \la I\ra_\alpha\}\cup\{x_1,\cdots,x_n\}\cup\{\tau_1,\cdots,\tau_{n-1}\}\] subject to the following relations
   \begin{itemize}
   \item  $ x_k x_l=x_l x_k$;
   \item the elements $\{e_{\textbf{i}}\mid \textbf{i}\in\langle I\rangle_\alpha\}$ are mutually orthogonal idempotents whose sum is the identity $e_\alpha\in R_\alpha$;
   \item $x_k e_{\textbf{i}}=e_{\textbf{i}}x_k$ and $\tau_k e_{\textbf{i}}=e_{t_k\textbf{i}}\tau _k$;
   \item $(\tau_k x_l-x_{t_k(l)}\tau_k)e_{\textbf{i}}=\delta_{i_k,i_{k+1}}(\delta_{k+1,l}-\delta_{k,l})e_\bfi$
   \item $\tau_k^2 e_{\textbf{i}}=q_{i_k,i_{k+1}}(x_k,x_{k+1})e_\bfi$

   \item $\tau_k \tau_l =\tau_l \tau_k \  \text{if $\mid k-l\mid> 1$} $
   \item $(\tau_{k+1}\tau_k\tau_{k+1}-\tau_k\tau_{k+1}\tau_k)e_{\textbf{i}}=\delta_{i_k,i_{k+2}}\frac{q_{i_k,i_{k+1}}(x_k,x_{k+1})-q_{i_k,i_{k+1}}(x_{k+2},x_{k+1})}{x_k-x_{k+2}}e_\bfi$
 \end{itemize}
 where $t_k$ is the simple translation of $n$ such that $t_k(k)=k+1,t_k(k+1)=k$ and $t_k(l)=l$ for $l\neq k,\, k+1$.
 \end{Definition}
 We remark that $R_\alpha$ is an associated $\ZZ$-graded algebra by setting $\deg x_i=2$ and $\deg \tau_k e_\bfi=-(\alpha_{i_k},\alpha_{i_{k+1}})$. Therefore, we can consider the $\ZZ$-graded modules over $R_\alpha$. For a locally finite dimensional module $M$, one defines $$\Dim M=\sum_{i\in\ZZ}\odim M_i q^i$$
 And set $M[1]$ so that $M[1]_i=M_{i-1}$ for all $i\in \ZZ$, which leads to the formula $$\Dim M[1]=q\Dim M$$

\subsection{Geometrization of KLR algebras}
Following~\cite{VV}, we explore the geometrization of KLR-algebras.  Assuming $\alpha$ be a dimension vector in $\Qp$ and recalling the Lusztig's sheaves \refb{eq:BBDlusztig}, we define
$$\cL_\alpha=\bigoplus_{\bfi\in \wI{\alpha}} \cL_\bfi$$
We then introduce the notion $R(\alpha)=\Ext_{\GL{(\alpha)}}^\bullet(\cL_\alpha,\cL_\alpha)$. For further details on these concepts, we refer the reader to \cite[Section 1.2]{VV}. It is a known fact that $R(\alpha)$ is a KLR algebra associated with $Q$. Since
\begin{equation}\label{eq:BDDalpha}
  \cL_\alpha=\bigoplus_{\lambda\in \cP_\alpha}L(\lambda)\boxtimes \IC(\lambda)
\end{equation}
 for some $\ZZ$-graded vector spaces $L(\lambda)$, then these $L(\lambda)$ can be thought of as modules over $R_\alpha$. The following theorem shows that all the simple modules over $R_\alpha$ are of this form (up to some scaling). 
\begin{Remark}\label{rem:selfdual}
Note that all of $L(\lambda)$ are self-dual simple modules, Since $\cL_\alpha$ and $\IC(\lambda)$ are $\DD$-invariant, where $\DD$ refers to Verdier duality.
\end{Remark}
\begin{Theorem}[{\cite[Theorem 8.6.12]{CG}}]\label{them:simplemodulesKLR}
The non-zero terms of collection $\{L(\lambda)\}$ arising form \refb{eq:BDDalpha} form a complete set of the isomorphism class of self-dual simple $R(\alpha)$-modules.
\end{Theorem}
\begin{proof}
 By~\cite[Chapter 8]{CG} we have $$\Ext_{G(\alpha)}^\bullet(L_\alpha,L_\alpha)=\bigoplus_{\lambda\in\cP_\alpha}\End(L(\lambda))\bigoplus_{\lambda,\gamma\in \cP_\alpha} \Hom(L(\lambda),L(\gamma))\Ext_{G(\alpha)}^{\geq 1}(\IC(\lambda),\IC(\gamma))$$ 
 Thus, the simple modules are of the form $L(\lambda)$, where $\lambda\in \cP_\alpha=\KP{\alpha}$.
\end{proof}

\subsection{Induction and Restriction Functors}
Let $\PP_\lambda$ be the projective cover of $L(\lambda)$ which are of the form $\Ext_{G_\alpha}^\bullet(\cL_\alpha,\IC(\lambda))$. Let $\PP_\bfi=\Ext_{G_\alpha}^\bullet(\cL_\alpha,\cL_\bfi)$ be the projective module associated with word $\bfi$. Next, we will define the product of projective modules over $R_Q=\bigoplus\limits_{\alpha\in\Qp}R(\alpha)$.
\begin{Definition}\label{def:productresprojectivemodule}
Let $\gamma=\apb$ be a partition of $\gamma$, and $\munu$ be two Kostant partitions in $\KP{\alpha}$ and $\KP{\beta}$. Let $\PP_\mu$ and $\PP_\nu$ be the indecomposable projective modules associated with $\munu$. One defines \emph{Induction Functor} by 
\begin{equation}
  \PP_\mu\star \PP_\nu=\Ext_{G_\gamma}^\bullet(\cL_\gamma,\ICmu\star\ICnu)
\end{equation}
Let $\lambda\in \KP{\gamma}$ and $\PP_\lambda$ be the indecomposable projective modules associated with $\lambda$. One defines \emph{Restriction Functor} by
\begin{equation}\label{eq_resfunPP}
  \Res_{\adb}^\gamma(\PP_\lambda)=\Ext_{G_\alpha\times G_\beta}^\bullet(\cL_\alpha\boxtimes\cL_\beta,\Res_{\adb}^\gamma\IClam)
\end{equation}
\end{Definition}

\begin{Remark}\label{rem:product}
In~\cite[Section 4.6]{VV} and~\cite[Theorem 3.1]{McN}, the above definition is equivalent to the Induction and Restriction functors given in~\cite[Section 3.1]{KL09}.
\end{Remark}
We denote by $\Rproj$ the category of the graded projective modules over $R_Q$ and by $K^0(R)$ its Grothendieck group.  We denote by $\Rfin$ the category of the graded finite dimensional modules over $R_Q$ and by $\KR$ its Grothendieck group. It can be observed that $\KR$ is generated by simple module $L(\lambda)$ for all $\lambda\in\KP{\alpha}$.

For any projective module $\PP$, one defines its \emph{dual} as $\PP^\sharp=\Hom_{R(\alpha)}(\PP,R(\alpha))$. In \cite{KL09} they demonstrated the existence of an isomorphism $\gamma: K^0(R)\eqm\Una$, where $\Una$ is the half-part of the quantum group induced from the underlying graph of $Q$. \\ 

\noindent
The functor $\Ext_{G_\alpha}^\bullet(\cL_\alpha,-)$ gives the isomorphism 
\begin{equation}\label{eq_isoprojective}
\begin{split}
  \Ext_{G_\alpha}^\bullet(\cL_\alpha,-):\,  K_0(\cQ_\alpha)\eqm &K^0(R)\\ 
       \cL \mapsto &\Ext_{G_\alpha}^\bullet(\cL_\alpha,\cL)
\end{split}
\end{equation}
For any $\cL\in \cQ_\alpha$ as in Section \ref{sec_lus}, we have 
\begin{equation}
  \Ext_{G_\alpha}^\bullet(\cL_\alpha,\DD(\cL))=\PP_{\cL}^\sharp
\end{equation}
where $\PP_{\cL}$ refers to the projective module $\Ext_{G_\alpha}^\bullet(\cL_\alpha,\cL)$.\\

\noindent
There is a non-degenerate paring as follows
\begin{equation}\label{eq:k0paring}
\begin{split}
\lara{-}{-}:\, K^0(R)\times \KR&\to \ZZ[q,q^{-1}]\\
(\PP \  ,\  M)&\mapsto \Dim \hom_{R}(\PP,M)
\end{split}
\end{equation}
where $\Dim \hom(\PP,M)=\sum\limits_{n\in \ZZ}\odim \Hom_R(\PP,M[n])q^n$
\begin{Lemma}\label{lem:coefficientsimplemodules}
Let $\alpha\in\Qp$, $\cL_a\in \add{\cP_\alpha}$, and $\PP_a=\Conv{\cL_a}$. If we decompose $\cL_a$ as 
$$\cL_a=\bigoplus_{\mu\in \cP_\alpha}L_a(\mu)\boxtimes \IC(\mu)$$
then we have $\Dim \hom(\PP_a,L(\mu))=\Dim L_a(\mu)$ where $\Dim L_a(\mu)=\sum\limits_{n\in \ZZ}\odim L_a(\mu)_n q^n$
\end{Lemma}
\begin{proof}
First, we assume that $\PP_a$ is indecomposable projective modules. For any $\mu'\in\KP{\alpha}$, we have
$$\Hom_R(\PP_{\mu'},L(\mu))=\delta_{\mu',\mu}$$
Since $\PP_a\cong\displaystyle\bigoplus_{\mu\in \cP}\PP_\mu^{\Dim L_a(\mu)}$, then 
$$\Hom_R(\PP_a,L(\mu))=\Hom_R(\PP_\mu^{\Dim L_a(\mu)},L(\mu))=\Dim L_a(\mu)$$
\end{proof}
\noindent
Next, we define the product on $\KR$.
\begin{Definition}\label{def:productresfinitemodule}
For two modules $M,N\in \Rfin$, we define their product $M\circ N$ so that
\begin{equation}
  \la\PP,M\circ N\ra=\la \Res\PP, M\boxtimes N\ra  \quad \text{for any projective modules $\PP\in \Rproj$.}
\end{equation}
where $\Res$ is as in \refb{eq_resfunPP}.
Since this pairing is non-degenerate, we get $M\circ N$ is well-defined (up to isomorphism). 
\end{Definition}
\begin{Remark}
The above definition coincides with the Induction functor given in~\cite[Section 3.1]{KL09} by~\cite[Propositional 3.3]{KL09}. Let $\apb=\gamma$ be a partition of $\gamma$. There exists a canonical embedding 
\[R_\alpha\otimes R_\beta\to R_\gamma\]
We denote by $e_{\adb}$ the image of $\sum\limits_{\bfi\in \wI{\alpha};\bfj\in\wI{\beta}}e_\bfi e_\bfj$. The induction of two modules $M,N$ is given by $M\circ N=R_\gamma e_{\adb}\otimes_{R_\alpha\otimes R_\beta}M\boxtimes N$. For any $\bfi\in \wI{\gamma}$ we have 
\begin{align*}
e_\bfi R_\gamma e_{\adb}&=\hom(\PP_\bfi, \PP_\adb)\\
&=\Ext_{G_\gamma}^\bullet(\cL_\bfi,\cL_\alpha\star\cL_\beta) \quad &\text{ by isomorphism \refb{eq_isoprojective}}\\
&=\Ext_{G_\gamma}^\bullet(\Res_{\adb}\cL_\bfi,\cL_\alpha\boxtimes\cL_\beta) &\text{ by the adjoint property of $(\Res, \star)$}\\
&=\hom(\Res_{\adb}\PP_\bfi,R_\alpha\otimes R_\beta) &\text{ by isomorphism \refb{eq_isoprojective}}
\end{align*}
Here $\hom(A,B)=\oplus_{n\in\ZZ}\Hom_{R_Q}(A,B[n])$. Hence, we have 
\begin{align*}
 \hom(\PP_\bfi,M\circ N)&=e_\bfi R_\gamma e_{\adb}\otimes_{R_\alpha\otimes R_\beta}(M\boxtimes N)\\
 &=\hom(\Res_{\adb}\PP_\bfi,R_\alpha\otimes R_\beta)\otimes_{R_\alpha\otimes R_\beta}(M\boxtimes N)\\
 &=\hom(\Res_{\adb}\PP_\bfi,M\boxtimes N)
\end{align*}
Let us consider the functor $\Res_{\adb}(M):= e_{\adb}M$. Since $$\hom(\PP_\bfj\boxtimes\PP_\bfk,\Res_{\adb}(M))=e_\bfj e_\bfk M=\hom(\PP_\bfj\star\PP_\bfk,M)=\hom(\PP_{\bfj\bfk},M)$$ then for any two projective modules $\PP_1,\PP_2$, we obtain
\begin{equation}
  \hom(\PP_1\boxtimes\PP_2,\Res_{\adb}(M))=\hom(\PP_1\star\PP_2,M)
\end{equation}
\end{Remark}

\subsection{Two product of simple modules}
We will give an interpretation of the Jordan-Holder filtration of the Induction of two simple modules over $R_Q$ using the Decomposition \refb{eq:ResP}.
\begin{Lemma}\label{lem:multiplicitycomputer}
For $\gamma\in \Qp$ and a given $\cL_a\in\cQ_\gamma$, let $\PP_a$ be its corresponding projective module. Suppose that $$\Res_{\adb}^\gamma(\cL_a)=\sum\limits_{\mu\in\KP{\alpha},\nu\in\KP{\beta}}J_a(\munu)(\IC_\mu\boxtimes \IC_\nu)$$
Then we get
$$\la\PP_a,\, L(\mu)\circ L(\nu)\ra=\Dim J_a(\mu,\nu)$$
\end{Lemma}
\begin{proof}
Recall $\la\PP_a,\, L(\mu)\circ L(\nu)\ra=\la \Res\PP_a, L(\mu)\boxtimes L(\nu)\ra$. By Definition~\refb{def:productresprojectivemodule}, we get 
 \begin{align*}
 &\la \Res\PP_a, L(\kappa)\boxtimes L(\nu)\ra\\
 =&\left\la \sum\limits_{\mu'\in\KP{\alpha},\nu'\in\KP{\beta}}J_a(\mu',\nu')\boxtimes (\PP_{\mu'}\otimes\PP_{\nu'}),L(\mu)\boxtimes L(\nu)\right\ra\\
 =&J_a(\mu,\nu)
 \end{align*}
 \end{proof}
\noindent
 Since $L(\mu)\circ L(\nu)$ is finite dimensional, it admits a Jordan-H\"older filtration. We will show that the multiplicity of $L(\lambda)$ in this filtration, denoted by $[L(\mu)\circ L(\nu)):L(\lambda)]$, is equal to $\Dim\hom_R(\PP_\lambda, L(\mu)\circ L(\nu))$. Namely,
\begin{equation}\label{eq_dimhomp}
  \Dim\hom_R(\PP_\lambda, L(\mu)\circ L(\nu))=[L(\mu)\circ L(\nu):L(\lambda)]
\end{equation}
   We show this by induction on the length of this filtration. If $L(\mu)\circ L(\nu)$ is a simple module, there is nothing to prove. Assume this statement is true for modules $M$ with shorter filtrations compared to the filtration length of $L(\mu) \circ L(\nu)$. Considering the following short exact sequence:
 \[0\to M\to L(\mu)\circ L(\nu)\to N\to 0\]
 and then applying the functor $\hom_R(\PP_\lambda,-)$ to it, we obtain the following short exact sequence:
 \[0\to \hom_R(\PP_\lambda,M)\to \hom_R(\PP_\lambda,L(\mu)\circ L(\nu))\to \hom_R(\PP_\lambda,N)\to 0\]
 This means 
 \[\Dim\hom_R(\PP_\lambda, L(\mu)\circ L(\nu))=\Dim\hom_R(\PP_\lambda, M)+\Dim\hom_R(\PP_\lambda, N)\]
 On the other hands, the multiplicity of $L(\lambda)$ in $L(\mu)\circ L(\nu)$ is equal to 
 \[[L(\mu)\circ L(\nu):L(\lambda)]=[M:L(\lambda)]+[N:L(\lambda)]\]
According to our hypothesis, it can be inferred that $[M:L(\lambda)]=\Dim\hom_R(\PP_\lambda, M)$, $[N:L(\lambda)]=\Dim\hom_R(\PP_\lambda, N)$, and then $\hom_R(\PP_\lambda,L(\mu)\circ L(\nu))=[\simppro{\mu}{\nu}:L(\lambda)]$.\\ 

\noindent
 Using the above lemma, we get the following theorem. 
\begin{Theorem}\label{them:multiplicitycirc}
Let $\gamma=\apb$ be a partition of $\gamma$, $\lambda$ be a Kostant partition of $\gamma$, and $\munu$ be two Kostant partitions in $\KP{\alpha}$ and $\KP{\beta}$. We have 
$$[L(\mu)\circ L(\nu):L(\lambda)]=[\Res_{\adb}^\gamma\IC(\lambda):\IC(\mu)\boxtimes\IC(\nu)]=K_\lambda(\munu)$$ where $[\Res_{\adb}^\gamma\IC(\lambda):\IC(\mu)\boxtimes\IC(\nu)]$ refers to the multiplicity of $\IC(\mu)\boxtimes\IC(\nu)$ in $\Res_{\adb}^\gamma\IC(\lambda)$ as Lemma~\ref{lem:multiplicitycomputer}. In particular, for $\lambda\in \ext(\munu)$, we have that the top degree of $K_\lambda(\munu)$ is less than and equal to 
\[2e_\lambda(\munu)+d_\lambda(\munu)\]
If $(\munu)\in\ext_\adb^{ger}(\lambda)$, then 
\[K_\lambda(\munu)=q^{\epsilon(v_1,\alpha,v_2,\beta)}\Po{\Gr_{v_1-v_\mu}(CK(M_\mu))\times \Gr_{v_2-v_\nu}(CK(M_\nu))}\]
where $e_\lambda(\munu)$ is the dimension of $\cExt{\lambda}{\mu}{\nu}$, (see \ref{lem:dimfiberkappa}).
\end{Theorem}
\begin{proof}
 As previously mentioned, we have $[L(\kappa)\circ L(\nu)):L(\lambda)]=\Dim\hom_R(\PP_\lambda, L(\kappa)\circ L(\nu))$. By the above Lemma and Theorem \ref{pro_topdegree}, we get required results.
\end{proof}



\begin{Proposition}\label{pro:multiplicityRes}
Suppose that $Q$ is a quiver without loops, (not necessary for Dynkin quivers). Let $\apb=\gamma$ be a partition of $\gamma\in\Qp$, $\munu\in\KP{\alpha}\times\KP{\beta}$, and $\lambda\in\KP{\gamma}$. We have 
$$[\Res_{\adb}L(\lambda):L(\mu)\boxtimes L(\nu)]=[\IC(\mu)\star\IC(\nu):\IC(\lambda)]=J_\lambda(\kappa,\nu)$$ where $[\IC(\mu)\star\IC(\nu):\IC(\lambda)]$ refers to the multiplicity of $\IC(\lambda)$ in $\IC(\mu)\star\IC(\nu)$.
\end{Proposition}
\begin{proof}
 Recall $[\Res_{\adb}L(\lambda):L(\mu)\boxtimes L(\nu)]=\Dim \hom(\PP_\mu\boxtimes\PP_\nu, \Res_{\adb}L(\lambda))$ and $$\hom(\PP_\mu\boxtimes\PP_\nu, \Res_{\adb}L(\lambda))=\hom(\PP_\mu\star\PP_\nu, L(\lambda))$$ Let us decompose $\PP_\mu\star\PP_\nu=\bigoplus_{\lambda'\in\KP{\gamma}}J_{\lambda'}(\munu)\PP_{\lambda'}$. It follows that $\hom(\PP_\mu\star\PP_\nu, L(\lambda))=J_\lambda(\munu)$. The Definition \ref{def:productresprojectivemodule} implies that $[\IC(\mu)\star\IC(\nu):\IC(\lambda)]=J_\lambda(\munu)$, we finish our proof.
\end{proof}

\begin{Corollary}\label{cor:Ressimple0}
Under the assumptions as in Theorem \ref{them:multiplicitycirc}, we have that $\Res_{\adb}L(\lambda)\neq 0$ only if there exists a pair $\munu$ in $\KP{\alpha}$ and $\KP{\beta}$ such that $\lambda\geq \mu*\nu$.  
\end{Corollary}
\begin{proof}
 By the above Proposition and Theorem \ref{them:Indorbits}, we obtain our conclusion. Because we just consider the support of perverse sheaves $\IClam$, the conclusion of Theorem \ref{them:Indorbits} holds for any quivers $Q$ without loops. 
\end{proof}
\begin{Remark}
This Corollary is closely related to the notion \emph{semicuspidal} modules $L(\alpha)$ where $\alpha$ is a positive root. Suppose $\Res_{\munu}L(\alpha)\neq0$, then there exists a pair $(\munu)$ such that $\alpha\geq \mu*\nu$. On the other hand, $\alpha$ is the minimal Kostant partition in $\KP{\alpha}$. It follows that $\alpha=\mu*\nu$. Let $M_\alpha$ is indecomposable representation associated with $\alpha$. If there exists a pair $(\munu)$ such that 
\[0\to M_\nu\to M_\alpha\to M_\mu\to 0\]
It follows that $[M_\mu,M_\nu]^1\neq 0$. On the other hand, any summands in $M_\nu$ is greater than $\alpha$, as there is an injective map $M_\nu\to M_\alpha$ by \refb{eq_homext0}. Similarly, we have any summands in $M_\mu$ is less than $\alpha$. Therefore, $\Res_{\munu}L(\alpha)\neq0$ only if $\mu_i<\alpha$ and $\nu_j>\alpha$ for any $\mu_i\in \mu$ and $\nu_j\in\nu$. We give a new proof of McNamara’s Lemma, see \cite[Lemma 3.2]{McN4}. 
\end{Remark}

\begin{Corollary}\label{cor:head}
Given two simple modules $L(\mu)$ and $L(\nu)$ such that $\mu\in\KP{\alpha}$ and $\nu\in \KP{\beta}$, let us assume that the head of $L(\mu)\circ L(\nu)$ is a simple module and the socle of $L(\mu)\circ L(\nu)$ is a simple module, which are denoted by $L(\mu)\triangledown L(\nu)$ and $L(\mu)\triangle L(\nu)$ respectively. Denote by $\mu\triangledown\nu$ and $\mu\triangle\nu$ the Kostant partition associated with this head and socle respectively. We have
\begin{equation}
  \mu\triangledown\nu\geq \nu*\mu; \qquad   \mu\triangle\nu\geq \mu*\nu
\end{equation}
\end{Corollary}
\begin{proof}
Let us consider the case $L(\mu\triangle\nu)\to L(\mu)\circ L(\nu)$. It follows from adjoint functor that $\Res_{\beta,\alpha}L(\mu\triangle\nu)\to L(\nu)\boxtimes L(\mu)$. It implies that $J_{\mu\triangle\nu}(\nu,\mu)\neq 0$. By Corollary \ref{cor:Ressimple0} and Theorem \ref{them:multiplicitycirc}, we have $\mu\triangle\nu\geq \mu*\nu$.\\

\noindent 
In the case $L(\mu)\circ L(\nu)\to L(\mu\triangledown\nu)\to 0$. Applying the functor $(-)^*$, we obtain
\[0\to L(\mu\triangledown\nu)\to q^{(\adb)}\simppro{\nu}{\mu}\]
By the above conclusion, we have 
\[\mu\oplus\nu\geq\mu\triangledown\nu\geq \nu*\mu\] 
\end{proof}

\subsection{Support pairs}
In this section, we give the notion of a support pair. This will give a necessary condition for some conjectures introduced in \cite[Conjecture 1.3]{LM2} and \cite[Probelem 7.6]{KR}. However, we should remark that this condition is not sufficient to answer these conjectures. Because this notion just focuses on the supports of the perverse sheaves on representation spaces $E_\alpha$. This will omit some information about the decomposition \refb{eq:ResP}. One of the main reasons is that the closure of orbits in $E_\alpha$ is singular in general. The singularity makes this approach more difficult than we expect. Therefore, we limit our attention to the support of perverse sheaves $\IClam$. 
\begin{Definition}\label{def_supportpair}
  We call a pair $(\munu)\in\KP{\alpha}\times\KP{\beta}$ a \emph{support pair} if and only if for any $\lambda<\mu\oplus\nu$ we have $(\mu,\nu)\notin\ext_\adb^{ger}(\lambda)$. This notion is inspired by Theorem \ref{pro_topdegree}
\end{Definition}
\noindent
For \cite[Probelem 7.6]{KR}, we have the following theorem.

\begin{Theorem}\label{them:productsimple}
Let $L(\mu)$ and $L(\nu)$ be two simple modules over $R_Q$ where $\mu,\nu$ are their Kostant partitions, respectively. If $L(\mu)\circ L(\nu)$ is a simple module, then $(\munu)$ is a support pair. In other words, for any non-trivial extension $\lambda\in\ext(\munu)$, we have 
\[[M_\nu, M_\mu\oplus M_\nu]>[M_\nu,M_\lambda]\quad \text{ if $(\munu)$ is a generic pair for $\lambda$}\]
and 
\[[M_\mu,M_\mu\oplus M_\nu]>[M_\mu,M_\lambda]\quad \text{ if $(\nu,\mu)$ is a generic pair for $\lambda$}\]
\end{Theorem}
\begin{proof}
It is easy to see that $L(\mu)\circ L(\nu)$ is a simple module if and only if $K_\lambda(\munu)=0$ for all $\lambda<\mu\oplus\nu$. Theorem \ref{pro_topdegree} implies $(\munu)\notin \ext_\adb^{ger}(\lambda)$. Otherwise, $$K_\lambda(\munu)=q^{\epsilon(v_1,\alpha,v_2,\beta)}\Po{\Gr_{v_1-v_\mu}(CK(M_\mu))\times \Gr_{v_2-v_\nu}(CK(M_\nu))}\neq 0$$
By the definition of support pair, we have that 
\[[M_\nu, M_\mu\oplus M_\nu]>[M_\nu,M_\lambda] \quad \text{ if $(\munu)$ is a generic pair for $\lambda$}\] 
Because $L(\mu)\circ L(\nu)$ is a simple module if and only if $L(\nu)\circ L(\mu)$ is a simple module, we see that
\[[M_\mu,M_\mu\oplus M_\nu]>[M_\mu,M_\lambda]\]
\end{proof}

\noindent
Another useful application of this theorem is related to \cite[Conjecture 4.3]{VV1}.
\begin{Remark}
Recall the \cite[Conjecture 4.3]{VV1} or \cite{BZ}: for a pair of dual canonical base elements $(\bb^*(\mu),\bb^*(\nu))$
\begin{center}
  $\bb^*(\mu)\star\bb^*(\nu)\in q^v\bB^*$ iff $\bb^*(\mu)\star\bb^*(\nu)=q^{\epsilon(v_1,\alpha,v_2,\beta)}\bb^* (\mu\oplus\nu)$ iff $\bb^*(\nu)\star\bb^*(\mu)\in q^v\bB^*$.
\end{center}
The above theorem provides a new proof of that conjecture. This is due to the fact that the simple modules $L(\mu), L(\nu)$ are mapped to the dual canonical basis elements $\bb^*(\mu), \bb^*(\nu)$ through the isomorphism $K_0(R_Q) \cong U_q(\mathfrak{n})^*$ (refer to \cite{KL09}).  It is sufficient to show  $K_{\mu\oplus\nu}(\munu)=q^{\epsilon(v_1,\alpha,v_2,\beta)}$. It is straightforward to see this from Theorem \ref{them:multiplicitycirc} when we take $v_1=v_\mu$, $v_2=v_\nu$, and $M_\lambda=M_\mu\oplus M_\nu$.\\ 
\end{Remark}

\noindent
The following is another application of our results. 
\begin{Theorem}\label{theo_submodule}
Let $\apb=\gamma$ be a partition of $\gamma\in \Qp$ and $(\munu)\in\KP{\alpha}\times\KP{\beta}$. If $(\munu)\in \ext_\adb^{ger}(\mu*\nu)$ then $q^{n}L(\mu*\nu)$ is the socle of $L(\mu)\circ L(\nu)$ for some integer $n$.
\end{Theorem}
\begin{proof}
 For $\lambda=\mu*\nu$, we have $\cExt{\lambda}{\mu}{\nu}\cong\Hom_\Omega(U,W)$. It follows from the facts $\overline{\OO}_\lambda=E_\gamma$ and $\cExt{\lambda}{\mu}{\nu}=\kappa^{-1}(M_\mu,M_\nu)$.  It implies that $2e_\lambda(\munu)$ is maximal with respect to all extensions $M_{\lambda'}\in \extq{M_\mu}{M_\nu}$. If $(\munu)\in \ext_{\adb}^{ger}(\mu*\nu)$, we have $K_\lambda(\munu)\neq 0$ according to Theorem \ref{pro_topdegree}. 

 For others $\lambda'>\mu*\nu$, we see that the top degree of $K_{\lambda'}(\munu)$ is less than or equal to $$2e_{\lambda'}(\munu)+d_{\lambda'}(\munu)<2 e_\lambda(\munu)+d_\lambda(\munu)$$ This means the top degree of $K_{\mu*\nu}(\munu)$ is maximal for all Kostant partitions $\lambda'\in\ext(\munu)$. On the other hand, by Corollary \ref{cor:head}, we see that the Kostant partitions $\mu\triangle\nu$ of the socle of $\simppro{\mu}{\nu}$ satisfy $\mu\triangle\nu\in \ext(\munu)$. By \cite[Lemma 7.5]{McN3}, $K_{\mu\triangle\nu}(\munu)$ admits the maximal top degree for all $K_{\lambda'}(\munu)\neq 0$. This leads to $\mu\triangle\nu=\mu*\nu$. 
\end{proof}
\begin{Remark}
  This theorem provides a new viewpoint on \cite[Conjecture 1.3]{LM3}. We remark that the above theorem couldn't answer that conjecture.\\ 
\end{Remark}

\noindent
Thanks to Lapid, he provides a useful example as follows.
\begin{Example}
We remark that if $L(\mu*\nu)$ is a submodule of $\simppro{\mu}{\nu}$ the condition $(\munu)\in\ext_{\adb}^{ger}(\mu*\nu)$ doesn't hold in general. The counterexample is given by E. Lapid: In type $A_3$, let us consider the Kostant partitions $\mu=[2,1]=[2,2]+[1,1],\nu=[3,2]=[3,3]+[2,2]$. Namely, they are both semisimple representations of $A_3$. It is easy to see that $\Ch L(\mu)=[2,1]$ and $\Ch L(\nu)=[3,2]$, where $\Ch$ is the character of modules over $R_Q$ (see \cite{Br} for more details). It implies by shuffle products that  
\begin{align*}
  \Ch L(\mu)\circ \Ch L(\nu)=[2,1,3,2]+[2,3,1,2]+[3,2,2,1]
\end{align*}
By \cite[Example 5.2]{LM2}, the socle of $\simppro{\mu}{\nu}$ is $L([2,1,3,2])$. The representation associated with this segment is the direct sum of simple representations $S_2\oplus S_1$ and $S_3\oplus S_2$. On the other hand, $\ext_\adb^{ger}([2,1]*[3,2])=[2,3][1,2]$ by Example \ref{exa_A3}. Therefore, $[2,1][3,2]$ is not in $\ext_\adb^{ger}([2,1]*[3,2])$. 
\end{Example}

\subsection{In the case of rigid representations}
In this section, we will consider the rigid representations $M$ of $Q$, \emph{i.e.} $[M,M]^1=0$. 
\begin{Corollary}\label{them:rigidcricsimple}
Suppose $M_\mu$ and $M_\nu$ are rigid modules and one of $L(\mu),L(\nu)$ is real. We have that $\simppro{\mu}{\nu}$ is a simple module if and only if $[M_\mu,M_\nu]^1=[M_\nu, M_\mu]^1=0$ In other words $M_\mu\oplus M_\nu$ is a rigid module.
\end{Corollary}
\begin{proof}
If $\simppro{\mu}{\nu}$ is a simple module over $R_Q$, then $(\munu)$ is a support pair by Theorem \ref{them:productsimple}. That means for any $\lambda<\mu\oplus\nu$ $(\munu)\notin\ext_\adb^{ger}(\lambda)$. But $(\munu)$ is the unique minimal pair in $\KP{\alpha}\times\KP{\beta}$. This implies that no Kostant partition $\lambda$ satisfies $\lambda<\mu\oplus\nu$. Namely, $M_\mu\oplus M_\nu$ is a rigid representation of $Q$.

Suppose that $M_\mu\oplus M_\nu$ is a rigid representation of $Q$. It is easy to see 
\[(\munu)\in\ext_{\adb}^{ger}(\mu*\nu); \quad (\nu,\mu)\in \ext_{\beta,\alpha}^{ger}(\nu*\mu)\]
Since $[M_\mu,M_\nu]^1=[M_\nu,M_\mu]^1=0$, we have $M_{\mu*\nu}=M_{\nu*\mu}=M_\mu\oplus M_\nu$. It implies by Theorem \ref{theo_submodule} that $L(\mu\oplus\nu)$ is the socle of $L(\mu)\circ L(\nu)$ and it is also the socle of $L(\nu)\circ L(\mu)$. Applying the dual functor to $q^nL(\mu\oplus \nu)\hookrightarrow L(\nu)\circ L(\mu)$, we get $q^{n'}L(\mu\oplus\nu)$ is a head of $L(\mu)\circ L(\nu)$. By \cite[Lemma 3.2.3 (iV)]{KKKO}, we see that $L(\mu)\circ L(\nu)$ is a simple module. 
\end{proof}

In \cite{KKKO}, Kang, Kashiwara, Kim and Oh define $\mathfrak{d}(M,N)$ for any two modules $M,N$ over $R_Q$, which plays a key role in their works, see \cite[Definition 3.2.2]{KKKO}. We will calculate $\mathfrak{d}(L(\mu),L(\nu))$ for any two modules $L(\mu),L(\nu)$ such that their corresponding module $M_{\mu},M_{\nu}$ over $kQ$ are rigid. Let us first give a notation as follows. For any two elements $\alpha,\beta\in\ZZ[I]$, we define 
\[\alpha\cdot_\Omega\beta:=\sum_{h\in \Omega}\alpha_{s(h)}\beta_{t(h)}\]
\begin{Proposition}\label{pro:produclength2}
Let $M_\mu,M_\nu$ be two rigid modules over $Q$ and $L(\mu),L(\nu)$ be their corresponding modules over $R_Q$. Let us assume that the head and socle of $L(\mu)\circ L(\nu)$ are simple modules. We have 
$$\mathfrak{d}(L(\mu),L(\nu))=\odim \OO_{\mu*\nu}+\odim \OO_{\nu*\mu}-2(\alpha-\beta)\cdot_\Omega(\alpha-\beta)$$
\end{Proposition}
\begin{proof}
 As $M_\mu,M_\nu$ are rigid module, it follows $(\munu)\in\ext_\adb^{ger}(\lambda)$. By Theorem \ref{theo_submodule} and Theorem \ref{pro_topdegree}, we see that $q^{n_1}L(\mu*\nu)$ is a socle of $L(\mu)\circ L(\nu)$, where $n_1=2e_{\mu*\nu}(\munu)+d_{\mu*\nu}(\munu)$. Similarly, $q^{n_2}L(\nu*\mu)$ is a socle of $L(\nu)\circ L(\mu)$, where $n_2=2e_{\nu*\mu}(\nu,\mu)+d_{\nu*\mu}(\nu,\mu)$. Applying the dual functor on $q^{n_2}L(\nu*\mu)\hookrightarrow L(\nu)\circ L(\mu)$, we obtain 
 \[q^{(\adb)}L(\mu)\circ L(\nu)\twoheadrightarrow q^{-n_2}L(\nu*\mu)\] 
 It follows that $q^{-(\adb)-n_2}L(\nu*\mu)$ is a head of $L(\mu)\circ L(\nu)$. By \cite[Theorem 4.2.1; Corollary 4.2.3; Corollary 4.2.4]{KKKO}, we have
 \[\mathfrak{d}(L(\mu),L(\nu))=n_1+n_2+(\adb)\]
 Since $e_{\mu*\nu}(\munu)=\alpha\cdot_{\Omega}\beta$ and $e_{\nu*\mu}(\mu,\nu)=\beta\cdot_\Omega \alpha$, it follows 
 \begin{align*}
     \mathfrak{d}(L(\mu),L(\nu))=&(\alpha,\beta)+2(e_{\mu*\nu}(\munu)+e_{\nu*\mu}(\mu,\nu))+d_{\mu*\nu}(\munu)+d_{\nu*\mu}(\nu,\mu)\\
     =&(\alpha,\beta)+2(\alpha\cdot_\Omega\beta+\beta\cdot_\Omega\alpha)+\odim \OO_{\mu*\nu}+\odim \OO_{\nu*\mu}-(\adb)-2(\alpha\cdot_\Omega\alpha+\beta\cdot_\Omega\beta)\\
     =&\odim \OO_{\mu*\nu}+\odim \OO_{\nu*\mu}-2(\alpha-\beta)\cdot_\Omega(\alpha-\beta).
 \end{align*}
\end{proof}

\begin{Remark}
Because all indecomposable modules are rigid modules, it follows that cuspidal modules over $R_Q$ can be considered as a special case of the above proposition.  Hence, we generalize the recent results given by Kashiwara and Oh, see \cite[Main Theorem]{KO}. 
\end{Remark}

\end{document}